\documentclass[11pt,a4paper]{amsart}

\usepackage{graphicx}
\usepackage{enumerate}
\usepackage{amsmath, amssymb, amsthm, amscd}
\usepackage{mathtools}
\usepackage{xcolor}
\usepackage{caption}
\usepackage{subcaption}
\usepackage{multirow}
\usepackage{extarrows}
\usepackage{stmaryrd}
\usepackage{tikz-cd}
\usepackage[pdfencoding=auto, psdextra, hidelinks]{hyperref}

\theoremstyle{plain}
\newtheorem{theorem}{Theorem}[section]
\newtheorem{lemma}[theorem]{Lemma}

\newtheorem{proposition}[theorem]{Proposition}

\theoremstyle{definition}
\newtheorem{definition}[theorem]{Definition}

\theoremstyle{remark}
\newtheorem{remark}[theorem]{Remark}

\numberwithin{equation}{section}

\newcommand{\R}{\mathbb{R}}

\newcommand{\Cinf}{C_c^\infty}
\newcommand{\tdom}{(0,T)}
\newcommand{\odom}{\tdom\times\Omega}

\newcommand{\du}{\delta u}
\newcommand{\M}{\mathcal{M}} 
\newcommand{\E}{\mathcal{E}}
\newcommand{\veps}{\varepsilon}
\newcommand{\vrho}{\varrho}

\newcommand{\epso}{{ 0,\varepsilon}}
\newcommand{\half}{\frac{1}{2}}
\newcommand{\dx}{\mathrm{d}x}
\newcommand{\dt}{\mathrm{d}t}

\newcommand{\psig}{h_\gamma}

\newcommand{\tu}{\Tilde{{u}}}
\newcommand{\tv}{\Tilde{{v}}}
\newcommand{\X}{\mathcal{X}}

\newcommand{\delt}{\Delta t}
\newcommand{\delx}{\Delta x}
\newcommand{\dely}{\Delta y}

\newcommand{\abs}[1]{\left\lvert#1\right\rvert}

\newcommand{\Dt}{\partial_t}
\newcommand{\Dx}{\partial_x}
\newcommand{\Dy}{\partial_y}

\title[AP and Energy Stable Scheme]{An Asymptotic Preserving and Energy Stable Scheme for the Euler System with Congestion Constraint}

\author[Arun]{K.R.\ Arun}
\address{School of Mathematics, Indian Institute of Science Education
  and Research Thiruvananthapuram, Thiruvananthapuram 695551, India} 
\email{arun@iisertvm.ac.in, amoghk0720@iisertvm.ac.in, harihara22@iisertvm.ac.in}
\thanks{K.\,R.\,A.\ acknowledges the support from Science and Engineering Research Board, Department of Science \& Technology, Government of India through grant CRG/2021/004078.}

\author[Krishnamurthy]{A.\ Krishnamurthy}

\author[Maharna]{H.\ Maharna}



\date{\today}

\begin{document}

\begin{abstract}
    In this work, we design and analyze an asymptotic preserving (AP), semi-implicit finite volume scheme for the scaled compressible isentropic Euler system with a singular pressure law known as the congestion pressure law. The congestion pressure law imposes a maximal density constraint of the form $0\leq \vrho <1$, and the scaling introduces a small parameter $\veps$ in order to control the stiffness of the density constraint. As $\veps\to 0$, the solutions of the compressible system converge to solutions of the so-called free-congested Euler equations that couples compressible and incompressible dynamics. We show that the proposed scheme is positivity preserving and energy stable. In addition, we also show that the numerical densities satisfy a discrete variant of the constraint. By means of extensive numerical case studies, we verify the efficacy of the scheme and show that the scheme is able to capture the two dynamics in the limiting regime, thereby proving the AP property. 
\end{abstract}

\maketitle

\section{Introduction}

Several macroscopic models of transport involve steric restrictions which prevents the overlapping of individual constituents within the model. Such restrictions give rise to what is known as the congestion phenomenon, which can impose several physical restraints on the model. Examples of such constraints would be the maximal density constraint that occurs when modelling traffic or pedestrian flows \cite{BDL+08, BGP+17}, constraints on the concentrations of individual components when modelling multiphase flows \cite{BBC+00} or the flux constraints that occur for supply chain problems \cite{ADR06}. In the present work, we focus on the one-dimensional isentropic Euler system with a singular pressure law, known as the congestion pressure law, as a hydrodynamic model of transport. The system of equations reads
\begin{subequations}
\label{eqn:eul-sys}
    \begin{gather}
        \Dt\vrho_\veps + \Dx(\vrho_\veps u_\veps) = 0, \label{eqn:mss-bal} \\
        \Dt(\vrho_\veps u_\veps) + \Dx(\vrho_\veps u^2_\veps) + \Dx(p_\veps(\vrho_\veps)) = 0, \label{eqn:mom-bal} \\
        \vrho_\veps (0,\cdot) = \vrho_\epso,\quad u_\veps(0,\cdot) = u_\epso. \label{eqn:ini-data}
    \end{gather}
\end{subequations}
The symbols $\vrho_\veps = \vrho_\veps(t,x)$ and $u_\veps = u_\veps(t,x)$ stand for the density and velocity respectively. Here, $p_\veps(\vrho_\veps)$ is the congestion pressure and is defined by
\begin{equation}
    \label{eqn:cong-pres}
    p_\veps(\vrho_\veps) = \veps \biggl(\frac{\vrho_\veps}{1-\vrho_\veps}\biggr)^\gamma;\quad \gamma\in(1,3\rbrack.
\end{equation}
The pressure law \eqref{eqn:cong-pres} imposes a constraint on the density of the form $\vrho_\veps<1$, and the upper bound $\vrho^* = 1$ is known as the congestion density. The model \eqref{eqn:eul-sys}-\eqref{eqn:cong-pres} is known to support the formation of congested regions, i.e.\ the regions where the density is close to 1; see, e.g.\ \cite{BP21, DHN11, DMN+18, DNB+10}. The congestion pressure law can be viewed as the average of all the short-range repulsive forces acting at the microscopic level. Hence, the singularity in the pressure plays the role of a barrier and as such, does not allow any overlap between singular molecules. In order to control the strength of these repulsive forces and the stiffness of the density constraint, we introduce a small parameter $\veps>0$ as done in \cite{DHN11, DMN+18, PS22}. One can observe that $p_\veps(\vrho_\veps)\sim\vrho_\veps^\gamma$ if $\vrho_\veps\ll 1$ and $p_\veps(\vrho_\veps)\to\infty$ if $\vrho_\veps\to 1$. In addition, due to the re-scaling, the pressure is only activated in congested regions, i.e.\ regions where $\vrho_\veps\sim 1$, whereas it will be of order $O(\veps)$ in the low-density regions; see Figure~\ref{fig:rescale_p} for a plot of the pressure function $p_\veps$ for different values of $\veps$. This behaviour of the pressure induces the formation of two vastly differing dynamics, namely the dynamics in the congested region and the free dynamics in the low-density regions. 

\begin{figure}
    \centering
    \includegraphics[width = 0.5\textwidth]{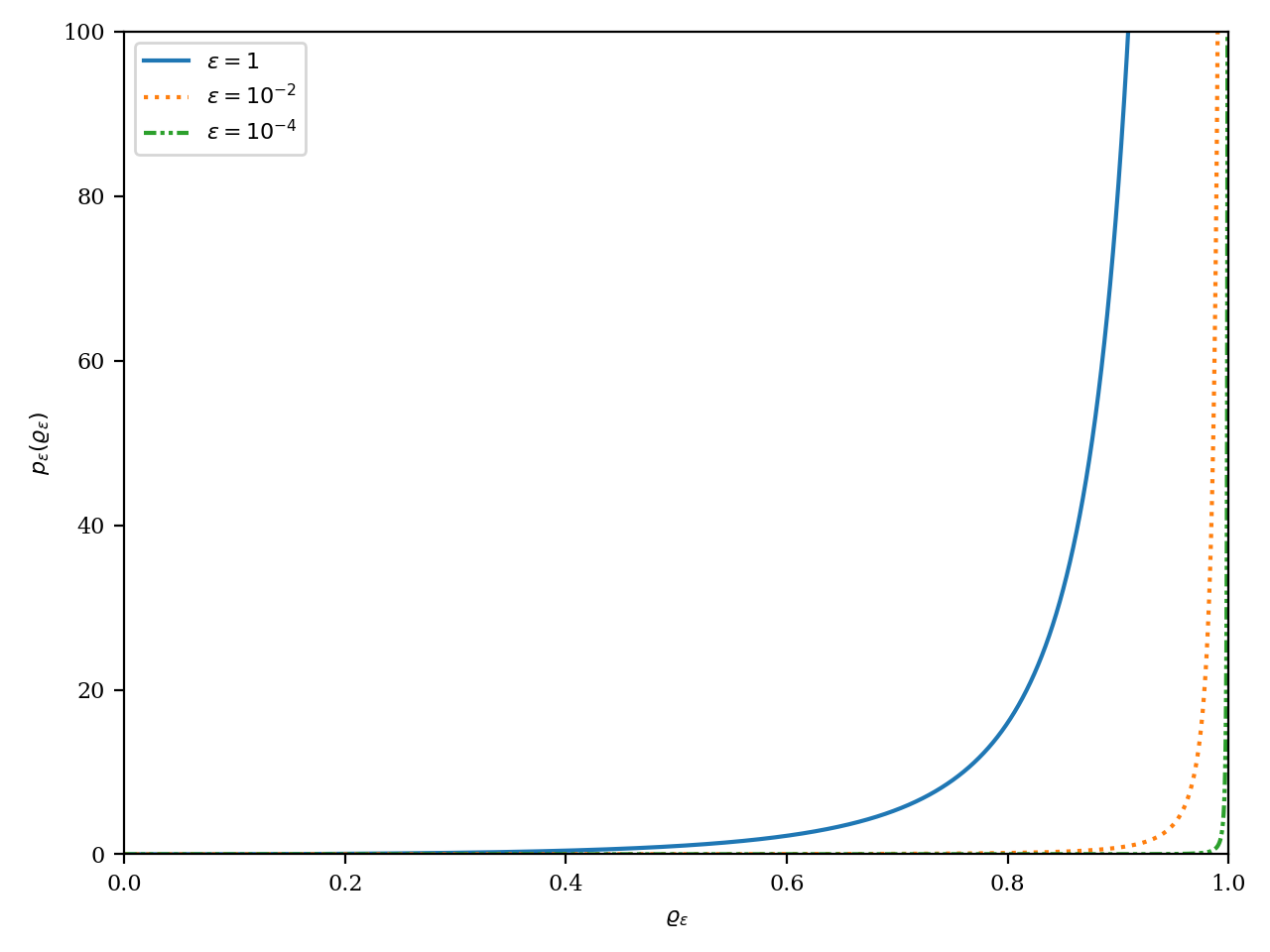}
    \caption{The re-scaled pressure $p_\veps(\vrho_\veps)$ in \eqref{eqn:cong-pres} for $\gamma = 2$ and different values of $\veps$.}
    \label{fig:rescale_p}
\end{figure}

Formally speaking, performing the limit $\veps\to 0$ of \eqref{eqn:eul-sys}-\eqref{eqn:cong-pres} leads to the following system of equations:  
\begin{subequations}
\label{eqn:lim-sys}
    \begin{gather}
        \Dt\vrho + \Dx(\vrho u) = 0, \label{eqn:mss-cong-lim} \\
        \Dt(\vrho u) + \Dx(\vrho u^2) + \Dx\pi = 0, \label{eqn:mom-cong-lim} \\
        0\leq\vrho\leq 1, \quad (1-\vrho)\pi = 0, \quad \pi\geq 0,\label{eqn:exclu-cong-lim} \\
        \vrho(0,\cdot) = \vrho_0,\quad u(0,\cdot) = u_0. \label{eqn:ini-data-cong-lim}
    \end{gather}
\end{subequations}
The above system is known as the free-congested one-dimensional Euler equations and was first introduced by Bouchut et al.\ \cite{BBC+00} as an asymptotic model for two phase flows; see also \cite{BP21, BPZ14, DHN11, DMN+18, PS22, PZ15} for more details. The limiting pressure $\pi$ is the limit of $p_\veps(\vrho_\veps)$ as $\veps\to 0$, and one can observe that if the limiting density $\vrho < 1$, then $\pi = 0$. Consequently, the limiting pressure is only activated in the congested regions where $\vrho = 1$ and hence, an exclusion constraint of the form $(1-\vrho)\pi = 0$ is part of the system \eqref{eqn:mss-cong-lim}-\eqref{eqn:ini-data-cong-lim}. If $\vrho<1$, we obtain the weakly hyperbolic system of equations of compressible, presureless dynamics and if $\vrho = 1$, we obtain the mixed hyperbolic-elliptic incompressible Euler system. Therefore, the limit system \eqref{eqn:mss-cong-lim}-\eqref{eqn:ini-data-cong-lim} is a hybrid system, coupling the free dynamics in the regions where $\vrho < 1$, with the congested, incompressible dynamics in the congested regions of $\vrho = 1$. Furthermore, since the original system \eqref{eqn:mss-bal}-\eqref{eqn:ini-data} is purely hyperbolic in nature, the limit $\veps\to 0$ is a singular limit as the system \eqref{eqn:mss-cong-lim}-\eqref{eqn:ini-data-cong-lim} is no longer hyperbolic.

The existence of solutions, both weak and classical, for the system \eqref{eqn:mss-bal}-\eqref{eqn:ini-data} with the pressure law given by \eqref{eqn:cong-pres} is still an open problem for the multi-dimensional case. In the one-dimensional setting, Bianchini and Perrin \cite{BP21} showed the existence of smooth solutions to the system \eqref{eqn:mss-bal}-\eqref{eqn:ini-data} written in Lagrangian co-ordinates, namely the $p$-system. They also perform the singular limit $\veps\to 0$ and show that the smooth solutions to the $p$-system will converge to a weak solution of the Lagrangian variant of \eqref{eqn:mss-cong-lim}-\eqref{eqn:ini-data-cong-lim}. However, for the viscous counterpart, namely the Navier-Stokes equations with congestion pressure, the existence of strong solutions in one dimension and their asymptotic limit was rigorously shown by Bresch et al.\ \cite{BPZ14}. Further, Perrin and Zatorska \cite{PZ15} proved the existence of weak solutions and also performed the limit $\veps\to 0$ in the multi-dimensional case, wherein the congestion constraint was dependent on space, i.e.\ $\vrho^* = \vrho^*(x)$. In \cite{DMZ18}, the authors studied a more general case wherein the congestion density was allowed to vary with space and time, i.e.\ $\vrho^* = \vrho^*(t,x)$, and satisfied a transport equation. With this transport equation now a part of the system, the authors show the existence of weak solutions and further perform the singular limit $\veps\to 0$.

As the limit $\veps\to 0$ is singular in nature, one needs to design a sufficiently robust numerical scheme that is able to accurately approximate the solutions of \eqref{eqn:eul-sys}-\eqref{eqn:cong-pres} uniformly  in $\veps>0$. The scheme should also be able to capture the regime changes which occur when $\veps\to 0$. In addition to these, as the limit system \eqref{eqn:lim-sys} couples compressible and incompressible dynamics together, the scheme should be able to automatically capture the shift between these dynamics when $\veps\sim 0$. This is challenging because there is no apriori information available as to how the phase transitions will occur along the boundary that separates a low density region and a congested region. A generic yet powerful framework which encompasses all the above required qualities is provided by the so-called asymptotic preserving (AP) methodology, which is proven to be an efficient platform in order to design numerical schemes that accurately capture singular limits; see, e.g\ \cite{AGK23, DHN11, DT11, HLS21, Jin99} and the references therein for further details. Briefly, a scheme is said to be AP if the scheme reduces to a consistent and stable approximation of the limiting equations when the singular perturbation parameter, in this case $\veps$, goes to 0, and if the stability constraints on the scheme are independent of $\veps$.

In the context of Euler equations with the congestion pressure law, Degond et al.\ \cite{DHN11} devised two different numerical schemes based on the local Lax-Friedrichs fluxes for the one-dimensional case, and verified the asymptotic properties of the scheme via various case studies. Also, Degond et al.\ \cite{DMN+18} considered the variable congestion dependent on space and time, i.e. $\vrho^* = \vrho^*(t,x)$ which satisfies a transport equation, and constructed first and second order finite volume schemes. For the viscous counterpart, Perrin and Saleh \cite{PS22} presented two staggered schemes, a fully implicit scheme and a pressure correction scheme, for the multidimensional Navier-Stokes equations with congestion pressure law. Both schemes are shown to be energy stable, i.e.\ the solutions generated by both schemes satisfy a discrete variant of the energy balance. In addition, they are also able to prove that the numerical densities generated by the schemes are bounded above by a constant $C < 1$, meaning the density constraint is respected at the discrete level.

In the present work, we focus on designing a semi-implicit, AP, energy stable and a structure-preserving staggered finite volume scheme for the Euler equations with the congestion pressure law. The key to achieving energy stability and the positivity of density is the introduction of a shifted velocity in the convective fluxes of the mass and momenta, proportional to the stiff pressure gradient. In addition, we show that the numerical densities generated by the scheme are strictly less than 1. Further, the AP property of the scheme is verified via extensive numerical case studies involving the congested as well as the low-density regimes.

The rest of the paper is structured as follows. In Section \ref{sec:apri-est-vel-stab}, we recall a few a priori energy estimates satisfied by smooth solutions of system \eqref{eqn:eul-sys} and further, also introduce the notion of velocity stabilization. The introduction of the finite volume scheme and its stability analysis is detailed in Section \ref{sec:fin-vol-scm}. The numerical results are presented in Section \ref{sec:num-res} and we conclude the paper with some remarks in Section \ref{sec:conc}.

\section{A priori Energy Estimates and Velocity Stabilization}
\label{sec:apri-est-vel-stab}

In this section, we recall the a priori energy estimates that are satisfied by classical solutions of \eqref{eqn:eul-sys}. In addition, we also introduce the concept of velocity stabilization, and show how it will help us achieve energy stability at the discrete level. To this end, we begin by defining the so-called pressure potential or the Helmholtz potential $\psig$, associated to the pressure $p_\veps$, such that
\begin{equation}
    z\psig^\prime(z) - \psig(z) = p_\veps(z),
\end{equation}
for $0\leq z<1$. One can easily infer that
\[
\psig^{\prime\prime}(z) = \frac{p^\prime_\veps(z)}{z}>0
\]
for any $0<z<1$. Hence, $z\mapsto\psig(z)$ is a convex function.

In addition, we also suppose that the initial data $(\vrho_\epso, u_\epso)$ to \eqref{eqn:eul-sys} satisfy 
\begin{equation}
\label{eqn:ini-data-hyp}
    \begin{split}
        &0\leq \vrho_\epso<1, \quad \overline{\vrho}_\epso\coloneqq \frac{1}{\abs{\Omega}}\int_\Omega\vrho_\epso\,\dx \leq \overline{\vrho} < 1, \\
        &\int_\Omega\biggl\lbrack\half\vrho_\epso(u_\epso)^2 + \psig(\vrho_\epso)\biggr\rbrack\,\dx\leq C
    \end{split}
\end{equation}
uniformly in $\veps$, where $\abs{\Omega}$ denotes the Lebesgue measure of $\Omega$ and $\overline{\vrho}, C>0$ are constants independent of $\veps$. 

\subsection{A priori energy estimates}
\label{subsec:apri-est}

We now recall the a priori energy estimates satisfied by smooth solutions of \eqref{eqn:eul-sys}.

\begin{proposition}[A priori energy estimates]
\label{prop:apri-eng-est}
    The smooth solutions of \eqref{eqn:eul-sys} satisfy
    \begin{enumerate}
        \item a renormalization identity: 
        \begin{equation}
        \label{eqn:renorm-idt}
		  \Dt(\psig(\vrho_\veps))+\Dx(\psig(\vrho_\veps)u_{\veps})+p_ \veps(\vrho_\veps) \Dx u_\veps=0; 
		\end{equation}
		
        \item a kinetic energy identity:  
        \begin{equation}
        \label{eqn:ke-idt}
            \Dt\biggl(\half\vrho_\veps (u_\veps)^2\biggr)+\Dx\biggl(\half\vrho_\veps(u_\veps)^2 u_\veps\biggr)+\Dx(p_\veps(\vrho_\veps))u_\veps=0; 
		\end{equation}
		
        \item a total energy identity:
		\begin{equation}
        \label{eqn:total-energy-idt}
		  \Dt E_\veps+\Dx((E_\veps+p_\veps(\vrho_\veps))u_\veps) =0,
		\end{equation}
		where $E_\veps=\half\vrho_\veps (u_\veps)^2+\psig(\vrho_\veps)$ is the total energy.
	\end{enumerate}
\end{proposition}

The existence of smooth, global in-time solutions to the Lagrangian variant of the Euler system \eqref{eqn:eul-sys} was shown in \cite{BP21} for $\veps>0$ fixed. In addition, the authors of \cite{BP21} also illustrate the existence of global entropy weak solutions of \eqref{eqn:eul-sys}. In what follows, we define global entropy weak solutions of \eqref{eqn:eul-sys} and recall the existence result from \cite{BP21}.

\begin{definition}[Entropy weak solutions]
\label{def:ent-wk-sol}
    Let $(\vrho_\epso, q_\epso)\in (L^\infty(\R))^2$, where $q_\epso = \vrho_\epso u_\epso$ denotes the momentum, satisfy 
    \begin{equation}
    \label{eqn:wk-ini-hyp}
        0\leq \vrho_\epso \leq 1 - C\veps^{\frac{1}{\gamma - 1}}, \quad q_\epso\leq (1- C\veps^{\frac{1}{\gamma - 1}})\vrho_\epso
    \end{equation}
    a.e.\ on $\R$, where $C>0$ is a constant independent of $\veps$. 

    A pair $(\vrho_\veps, q_\veps)$, with $\vrho_\veps(0,\cdot) = \vrho_\epso$ and $q_\veps(0,\cdot) = q_\epso$, is a global entropy weak solution of \eqref{eqn:eul-sys}-\eqref{eqn:cong-pres} if the following hold:
    \begin{itemize}
        \item $(\vrho_\veps, q_\veps)\in (L^\infty((0,\infty)\times\R))^2$. Further, there exists a constant $A_\veps>0$ such that 
        \[
            0\leq\vrho_\veps\leq A_\veps<1,\quad\abs{q_\veps}\leq A_\veps \vrho_\veps
        \]
        a.e. in $(0,\infty)\times\R$;

        \item The mass balance \eqref{eqn:mss-bal} and the momentum balance \eqref{eqn:mom-bal} are satisfied in the weak sense, i.e.\ for all $\varphi,\psi\in\Cinf((0,\infty)\times\R)$,
        \begin{subequations}
            \begin{align}
                &\int_0^\infty\int_\R\vrho_\veps\Dt\varphi\,\dx\dt + \int_0^\infty\int_\R q_\veps\Dx\varphi \,\dx\dt = -\int_\R\vrho_\epso\varphi(0,x)\,\dx; \\
                &\int_0^\infty \int_\R q_\veps\Dt\psi\,\dx\dt + \int_0^\infty\int_\R\biggl(\frac{q_\veps^2}{\vrho_\veps} + p_\veps(\vrho_\veps)\biggr)\Dx\psi\,\dx\dt = -\int_\R q_\epso\psi(0,x)\,\dx;
            \end{align}
        \end{subequations}

        \item The entropy/energy inequality is satisfied, i.e.\ for all $\varphi\in\Cinf((0,\infty)\times\R)$,
        \begin{equation}
        \label{eqn:ent-ineq}
            -\int_0^\infty\int_\R\biggl(\half\frac{q_\veps^2}{\vrho_\veps} + \psig(\vrho_\veps)\biggr)\Dt\varphi\,\dx\dt - \int_0^\infty\int_\R\biggl\lbrack\biggl(\half\frac{q_\veps^2}{\vrho_\veps} + \psig(\vrho_\veps) + p_\veps(\vrho_\veps)\biggr)\frac{q_\veps}{\vrho_\veps}\biggr\rbrack\Dx\varphi\,\dx\dt \leq 0.
        \end{equation}
    \end{itemize}
\end{definition}

\begin{theorem}[Existence of global weak solutions]
\label{thm:glob-weak-sol}
    Consider the isentropic Euler system \eqref{eqn:eul-sys} with the pressure law given by \eqref{eqn:cong-pres}. Let $(\vrho_\epso,q_\epso)\in (L^\infty(\R))^2$ satisfy \eqref{eqn:wk-ini-hyp}. Then, there exists a global entropy weak solution $(\vrho_\veps,q_\veps)$ to \eqref{eqn:eul-sys}-\eqref{eqn:cong-pres} in the sense of Definition \ref{def:ent-wk-sol}. Further, the following inequality holds:
    \begin{equation}
    \label{eqn:wk-den-bound}
        0\leq\vrho_\veps\leq 1 - C\veps^{\frac{1}{\gamma - 1}}
    \end{equation}
    a.e.\ in $(0,\infty)\times\R$, where $C>0$ is a constant independent of $\veps$.
\end{theorem}

\begin{remark}
    In accordance with Theorem \ref{thm:glob-weak-sol}, the existence of weak solutions to system \eqref{eqn:eul-sys}-\eqref{eqn:cong-pres} is only known globally, i.e.\ on $(0,\infty)\times\R$. As we are interested in designing a finite volume scheme, however, we work on $\odom$, where $T>0$ and $\Omega$ is an open, bounded subset of $\R$. Hence, from here on, we assume the existence of weak solutions to \eqref{eqn:eul-sys}-\eqref{eqn:cong-pres} in $\odom$.
\end{remark}

\subsection{Velocity stabilization}
\label{subsec:vel-stab}

We consider the following modified system.
\begin{subequations}
\label{eqn:vel-stab-eul-sys}
    \begin{align}
        &\Dt\vrho_\veps + \Dx(\vrho_\veps(u_\veps - \du_\veps)) = 0, \\ 
        &\Dt(\vrho_\veps u_\veps) + \Dx(\vrho_\veps u_\veps (u_\veps - \du_\veps)) + \Dx p_\veps(\vrho_\veps) = 0,
    \end{align}
\end{subequations}
where $\du$ is a shift in the velocity, which denotes the stabilization term.

Analogous to Proposition \ref{prop:apri-eng-est}, we can establish the following a priori estimates for the modified system \eqref{eqn:vel-stab-eul-sys}. 

\begin{proposition}[A priori estimates of the modified system]
\label{prop:apri-vel-stab}
 Let $\tu_\veps = u_\veps - \du_\veps$. The smooth solutions of \eqref{eqn:vel-stab-eul-sys} satisfy:
    \begin{enumerate}
        \item a renormalization identity: 
        \begin{equation}
        \label{eqn:renorm-idt-mod}
		  \Dt(\psig(\vrho_\veps))+\Dx(\psig(\vrho_\veps)\tu_\veps)+p_ \veps(\vrho_\veps) \Dx \tu_\veps = 0; 
		\end{equation}
		
        \item a kinetic energy identity:  
        \begin{equation}
        \label{eqn:ke-idt-mod}
            \Dt\biggl(\half\vrho_\veps (u_\veps)^2\biggr)+\Dx\biggl(\half\vrho_\veps(u_\veps)^2 \tu_\veps\biggr)+\Dx(p_\veps(\vrho_\veps))\tu_\veps = -\du_\veps\Dx p_\veps(\vrho_\veps); 
		\end{equation}
		
        \item a total energy identity:
		\begin{equation}
        \label{eqn:total-energy-idt-mod}
		  \Dt E_\veps+\Dx((E_\veps+p_\veps(\vrho_\veps))\tu_\veps) = -\du_\veps \Dx p_\veps(\vrho_\veps).
		\end{equation}
	\end{enumerate}
\end{proposition}

If we set $\du_\veps = \eta \Dx p_\veps(\vrho_\veps)$, where $\eta>0$ is a constant, in \eqref{eqn:total-energy-idt-mod}, one observes that we obtain
\[
\Dt E_\veps + \Dx((E_\veps + p_\veps(\vrho_\veps))\tu_\veps) = -\eta (\Dx p_\veps(\vrho_\veps))^2 \leq 0.
\]
Thus, we see that the introduction of an appropriately chosen shifted velocity indeed yields the energy inequality \eqref{eqn:ent-ineq}.

\section{Finite Volume Scheme}
\label{sec:fin-vol-scm}

\subsection{One-dimensional Finite Volume Scheme}
\label{subsec:1d-scheme}

We want to numerically approximate the velocity stabilized Euler system \eqref{eqn:vel-stab-eul-sys} on $\odom$. As $\Omega\subset\R$ is open and bounded, we suppose that $\Omega = (a,b)$ for convenience. We consider a discretization $0=t^{0}<t^{1}<\dots<t^n<\cdots<t_{N}=T$ of the time interval $\lbrack 0,T\rbrack$ and let $\delt=t^{n+1}-t^{n}$ for $n=0,1,\dots,N-1$ be the constant time-step. For the space discretization, we consider a one-dimensional MAC discretization $(\M, \E$), where the primal grid $\M$ contains the control volumes $K_{i}=[x_{i-\half},x_{i+\half}]$ such that $\overline{\Omega} = \cup_{i=1}^M K_i$. Here, $x_{\half} = a$ and $x_{M+\half} = b$. Here $\E$ is the collection of all endpoints of the control volumes. For $i=1,\dots, M-1$ and to each endpoint $x_{i+\half}\in\E$, we associate a dual cell $D_{i+\half} = \lbrack x_i, x_{i+1}\rbrack$, where $x_i = \half(x_{i-\half} + x_{i+\half})$. We also set $D_{\half} = \lbrack x_{\half}, x_1\rbrack$ and $D_{M+\half} = \lbrack x_M, x_{M+\half}\rbrack$. To simplify the presentation, we suppose that the space-step $\delx = x_{i+\half} - x_{i-\half} = x_{i+1} - x_{i}$ is constant. The discrete density along with discrete pressure are approximated over the cell centres $x_i$ of the primal cells $K_i$. The velocities are approximated at the centres of the dual cells $D_{i+\half}$ associated to the endpoints $x_{i+\half}$. The unknowns are thus the density $\vrho_i$ for $i = 1,\dots ,M$, and the velocities $u_{i+\half}$ for $i = 0,\dots, M$. We refer to Figure \ref{fig:1D-MAC} for a representation of the MAC grid as well as the arrangement of the variables.

\begin{figure}[htpb!]
	\centering
	\begin{tikzpicture}
		\draw [thick](-1.0,0) -- (9,0);
		\draw[black] (5.5,0.2) node[anchor=west]{$\cdots$};
		\draw[black] (5.5,-0.2) node[anchor=west]{$\cdots$};
		\draw[black] (1.5,0.2) node[anchor=west]{$\cdots$};
		\draw[black] (1.5,-0.2) node[anchor=west]{$\cdots$};
		\draw[black] (-1.0,0) circle (0pt) node[anchor=south]{$x_{\half}$};
		\draw[black] (9,0) circle (0pt) node[anchor=south]{$x_{M+\half}$};
		\draw[black] (7,0) circle (0pt) node[anchor=south]{$x_{M-\half}$};
		\filldraw[black] (0,0) circle (2pt) node[anchor=south]{$x_{1}$};
		\draw[black] (1.0,0) circle (0pt) node[anchor=south]{$x_{\frac{3}{2}}$};
		\filldraw[black] (4,0) circle (2pt) node[anchor=south]{$x_{i}$};
		\filldraw[black] (3,0) circle (0pt) node[anchor=south]{$x_{i-\half}$};
		\filldraw[black] (5,0) circle (0pt) node[anchor=south]{$x_{i+\half}$};
		\filldraw[black] (8,0) circle (2pt) node[anchor=south]{$x_{M}$};
		\draw [black, thick] (3, 0.08)--(3,-0.08);
		\draw [black, thick] (5, 0.08)--(5,-0.08);
		\draw [black, thick] (1.0, 0.08)--(1.0,-0.08);
		\draw [black, thick] (-1.0, 0.08)--(-1.0,-0.08);
		\draw [black, thick] (7, 0.08)--(7,-0.08);
		\draw [black, thick] (9, 0.08)--(9,-0.08);
		\draw[black] (0,-0.1)  node[anchor=north]{$\varrho_{1}$};
		\draw[black] (4,-0.1)  node[anchor=north]{$\varrho_{i}$};
		\draw[black] (8,-0.1)  node[anchor=north]{$\varrho_{M}$};
		\draw[black] (3,-0.1)  node[anchor=north]{$u_{i-\half}$};
		\draw[black] (5,-0.1)  node[anchor=north]{$u_{i+\half}$};
            \draw[black] (-1.0, -0.1) node[anchor=north]{$u_{\half}$};
            \draw[black] (1.0, -0.1) node[anchor=north]{$u_{\frac{3}{2}}$};
            \draw[black] (7, -0.1) node[anchor=north]{$u_{M-\half}$};
            \draw[black] (9, -0.1) node[anchor=north]{$u_{M+\half}$};
	\end{tikzpicture}
	\caption{Arrangement of the variables in one-dimensional grid}
        \label{fig:1D-MAC}
\end{figure}
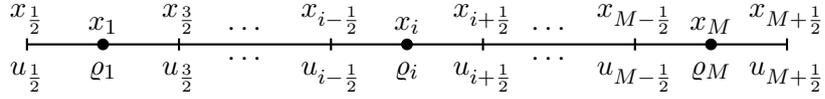

We initialize the scheme by setting 
\begin{equation}
\label{eqn:schm-init}
    \vrho^0_i = \frac{1}{\delx}\int_{K_i}\vrho_\epso \,\dx,\quad u^0_{i+\half} = \frac{1}{\delx}\int_{D_{i+\half}} u_\epso \,\dx.
\end{equation}
As a consequence of \eqref{eqn:ini-data-hyp}, note that $0\leq\vrho^0_i<1$. For the sake of simplicity, we omit the dependence of the discrete variables on $\veps$ throughout. Additionally, we also impose the boundary conditions to be periodic throughout.

The fully discrete, semi-implicit scheme for $0\leq n \leq N-1$ reads:
\begin{subequations}
\label{eqn:vel-stab-schm}
    \begin{align}
        &\frac{1}{\delt}(\varrho_{i}^{n+1}-\varrho_{i}^{n})+\frac{1}{\delx}\bigl(F_{i+\half}^{n+1}-F_{i-\half}^{n+1}\bigr)=0, \label{eqn:disc-mss-bal}   \\           
        &\frac{1}{\delt}(\varrho_{i+\half}^{n+1} u_{i+\half}^{n+1}-\varrho_{i+\half}^{n} u_{i+\half}^{n})+\frac{1}{\delx}(F_{i+1}^{n+1}u_{i+1}^{n}-F_{i}^{n+1}u_{i}^{n}) + 
        \eth_\E p^{n+1}_{i+\half} = 0. \label{eqn:disc-mom-bal}  
   \end{align}
\end{subequations}

The mass flux across the endpoints of the primal cells is defined as
\begin{equation}
\label{eqn:disc-mss-flx}
    F_{i+\half}^{n+1}= \varrho_{i}^{n+1}\tu_{i+\half}^{n,+}+\varrho_{i+1}^{n+1}\tu_{i+\half}^{n,-},
\end{equation}
where $\tu^n_{i+\half} = u^n_{i+\half} - \du^{n+1}_{i+\half}$ denotes the stabilized velocity. In accordance with the stability analysis done before, cf.\ \eqref{eqn:total-energy-idt-mod}, we set
\begin{equation}
\label{eqn:vel-stab-term}
    \du^{n+1}_{i+\half} = \eta\delt\eth_\E p^{n+1}_{i+\half}, 
\end{equation}
where $\eta > 0$ is to be determined later. Further, the positive and negative parts of the stabilized velocity read
\begin{align}
\label{eqn:pos-neg-vel}
    &\tu_{i+\half}^{n,+}=\max\{0,u_{i+\half}^{n}\}-\min\{0,\delta u_{i+\half}^{n+1}\} \geq 0,\\
    &\tu_{i+\half}^{n,-}=\min\{0, u_{i+\half}^{n} \}-\max\{0,\delta u_{i+\half}^{n+1} \} \leq 0.
\end{align}

The momentum convection flux is constructed using the mass flux and we define $F^{n+1}_i = \half(F^{n+1}_{i+\half} + F^{n+1}_{i-\half})$. The associated upwind velocity reads
\begin{equation}
	u_{i}^{n}= \begin{dcases}
		u_{i-\half}^{n}& ,\text{if } F^{n+1}_i\geq 0, \\
		u_{i+\half}^{n}& ,\text{otherwise}.
	\end{dcases}
\end{equation}
The discrete density at the dual cell centers, $\varrho_{i+\half}^n$, is defined as the average across its neighboring primal cells, i.e.\ $\varrho_{i+\half}^{n}=\half(\varrho_{i}^{n} + \varrho_{i+1}^{n})$. One can observe that the dual density $\vrho^{n}_{i+\half}$ will satisfy the following discrete dual mass balance, which can be obtained by considering the mass balance \eqref{eqn:disc-mss-bal} for indices $i$ and $i+1$, and then averaging the two.
\begin{equation}
\label{eqn:dual-mss-bal}
    \frac{1}{\delt}(\vrho^{n+1}_{i+\half} - \vrho^n_{i+\half}) + \frac{1}{\delx}(F^{n+1}_{i+1} - F^{n+1}_i) = 0.
\end{equation}

Finally, the operator $\eth_\E$ denotes a discretization of the space derivative on the dual cells for functions that are defined on the primal cells. Namely, we define 
\begin{equation}
\label{eqn:dual-der}
    \eth_\E p^{n+1}_{i+\half} =
    \begin{dcases}
        0,& \text{if } i =0, M,\\
        \frac{1}{\delx} (p^{n+1}_{i+1} - p^{n+1}_{i}),& \text{otherwise}.
    \end{dcases}
\end{equation}

In addition, for functions that are defined on the dual cells, we define a discretization of the space derivative on the primal cells denoted by $\eth_\M$ and we define it as 
\begin{equation}
\label{eqn:prim-der}
    \eth_\M w^n_i = \frac{1}{\delx} (w^n_{i+\half} - w^n_{i-\half}).
\end{equation}

One can observe that the following gradient-divergence duality relation holds:
\begin{equation}
\label{eqn:disc-grad-div}
    \sum_{i=1}^M \delx p^{n+1}_i \eth_\M w^n_i + \sum_{i=0}^M \delx w^n_{i+\half}\eth_\E p^{n+1}_{i+\half} = 0.
\end{equation}

Utilizing the dual mass update \eqref{eqn:dual-mss-bal} in the discrete momentum balance \eqref{eqn:disc-mom-bal}, one obtains the following discrete velocity update:
\begin{equation}
\label{eqn:disc-vel-update}
	\frac{\varrho_{i+\half}^{n+1}}{\delt}(u_{i+\half}^{n+1}-u_{i+\half}^{n}) + \frac{1}{\delx}\lbrack F_{i+1}^{n+1,-}(u_{i+\frac{3}{2}}^{n} -u_{i+\half}^{n}) -F_{i}^{n+1,+}(u_{i-\half}^{n} - u_{i+\half}^{n})\rbrack + \eth_\E p^{n+1}_{i+\half} = 0.
\end{equation}
Here, $F^{n+1,+}_i, F^{n+1,-}_{i+1}$ denote the standard positive and negative parts of a real number, i.e. $r^{\pm} = \half(r\pm\abs{r})$ so that $r^+\geq 0$ and $r^{-}\leq 0$.

\subsubsection{Energy stability of the scheme}

We now aim to prove the discrete variants of the identities \eqref{eqn:renorm-idt-mod}-\eqref{eqn:total-energy-idt-mod}. To that end, let $\llbracket q,r\rrbracket$ denote the interval $\lbrack\min\lbrace q,r\rbrace,\max\lbrace q,r\rbrace \rbrack$ for any $q,r\in\R$. As done in the continuous setting, we suppose that the total energy of the system at the discrete level remains bounded initially, i.e.\ there exists a constant $C>0$ independent of $\veps$ such that 
\begin{equation}
\label{eqn:disc-ini-enrg-hyp}
    \sum_{i=1}^M\delx\psig(\vrho^0_i) + \sum_{i=0}^M\delx\half\vrho^{0}_{i+\half}(u^0_{i+\half})^2\leq C.
\end{equation}

\begin{proposition}[Discrete energy identities]
    The solutions $(\vrho^n, u^n)_{0\leq n\leq N-1}$ generated by the scheme \eqref{eqn:vel-stab-schm} will satisfy
    \begin{enumerate}
        \item a discrete renormalization identity:
        \begin{equation}
        \label{eqn:disc-renorm}
            \frac{1}{\delt}(\psig(\vrho^{n+1}_i) - \psig(\vrho^n_i)) + \frac{1}{\delx}(Q^{n+1}_{i+\half} - Q^{n+1}_{i-\half}) + p^{n+1}_i \eth_\M\tu^n_i + R^{n+1}_i= 0,
        \end{equation}
        where $Q^{n+1}_{i+\half} = \psig(\vrho^{n+1}_i)\tu^{n,+}_{i+\half} + \psig(\vrho^{n+1}_{i+1})\tu^{n,-}_{i+\half}$. The remainder term $R^{n+1}_i$ is given as
        \begin{equation}
        \label{eqn:disc-renorm-rem}
            \begin{split}
                R^{n+1}_i &= \frac{1}{2\delt}(\vrho^{n+1}_i - \vrho^n_i)^2 \psig^{\prime\prime}(\overline{\vrho}^{n+\half}_i) + \frac{1}{2\delx}(\vrho^{n+1}_{i+1}-\vrho^{n+1}_i)^2 \psig^{\prime\prime}(\Tilde{\vrho}^{n+1}_{i+\half})(-\tu^{n,-}_{i+\half}) \\
                &+\frac{1}{2\delx}(\vrho^{n+1}_{i}-\vrho^{n+1}_{i-1})^2 \psig^{\prime\prime}(\Tilde{\vrho}^{n+1}_{i-\half})\tu^{n,+}_{i-\half},
            \end{split}
        \end{equation}
        where $\overline{\vrho}^{n+\half}_i\in\llbracket \vrho^n_i,\vrho^{n+1}_i\rrbracket$, $\Tilde{\vrho}^{n+1}_{i+\half}\in\llbracket\vrho^{n+1}_i,\vrho^{n+1}_{i+1}\rrbracket$ and $\Tilde{\vrho}^{n+1}_{i-\half}\in\llbracket\vrho^{n+1}_{i-1}, \vrho^{n+1}_i\rrbracket$;

        \item a kinetic energy identity:
        \begin{equation}
        \label{eqn:disc-ke}
            \begin{split}
                \frac{1}{\delt}\biggl(\half\vrho^{n+1}_{i+\half}(u^{n+1}_{i+\half})^2 - \half\vrho^n_{i+\half}(u^n_{i+\half})^2\biggr) &+ \frac{1}{\delx}\biggl(F^{n+1}_{i+1}\frac{(u^n_{i+1})^2}{2} - F^{n+1}_i\frac{(u^n_i)^2}{2}\biggr) \\
                + \tu^n_{i+\half}\eth_\E p^{n+1}_{i+\half} &+ S^{n+1}_{i+\half} = -\du^{n+1}_{i+\half}\eth_\E p^{n+1}_{i+\half},
            \end{split}
        \end{equation}
        where the remainder term $S^{n+1}_{i+\half}$ is given by 
        \begin{equation}
        \label{eqn:disc-ke-rem}
            \begin{split}
                S^{n+1}_{i+\half} = -\frac{1}{2\delt}\vrho^{n+1}_{i+\half}(u^{n+1}_{i+\half} - u^n_{i+\half})^2 - \frac{1}{2\delx} F^{n+1, -}_{i+1} (u^n_{i+\frac{3}{2}} - u^n_{i+\half})^2 + \frac{1}{2\delx} F^{n+1, +}_i (u^n_{i-\half} - u^n_{i+\half})^2.
            \end{split}
        \end{equation}
    \end{enumerate}
\end{proposition}

\begin{proof}
    The renormalization identity \eqref{eqn:disc-renorm} can be obtained by multiplying the discrete mass balance \eqref{eqn:disc-mss-bal} with $\psig^\prime(\vrho^{n+1}_i)$ and employing the Taylor expansions in time and space. The kinetic energy identity \eqref{eqn:disc-ke} is obtained by multiplying the velocity balance \eqref{eqn:disc-vel-update} with $u^n_{i+\half}$, using the identity $(q-r)r = \half (q^2 - r^2 - (q-r)^2)$ and using the dual mass balance \eqref{eqn:dual-mss-bal} to simplify. For further details, see the appendix of \cite{AA24}.
\end{proof}

Utilizing the above established identities, we can now show the energy stability of the scheme.

\begin{theorem}[Energy stability]
\label{thm:loc-ent-ineq}
    Assume that the following conditions hold for each $0\leq i\leq M$:
    \begin{enumerate}
        \item $\displaystyle\eta-\frac{3}{2\vrho^{n+1}_{i+\half}}>0$;
        \item $\displaystyle\frac{\delt}{\delx}\leq\frac{\vrho^{n+1}_{i+\half}}{3}\min\biggl\lbrace\frac{-1}{F^{n+1, -}_{i+1}},\frac{1}{F^{n+1, +}_i}\biggr\rbrace$.
    \end{enumerate}
    Then, for each $0\leq n\leq N-1$, the following local in-time energy inequality holds:
    \begin{equation}
    \label{eqn:disc-loc-ent-ineq}
        \sum_{i=1}^M\delx\psig(\vrho^{n+1}_i) + \sum_{i=0}^{M}\delx\half\vrho^{n+1}_{i+\half}(u^{n+1}_{i+\half})^2 \leq \sum_{i=1}^M\delx\psig(\vrho^n_i) + \sum_{i=0}^M\delx\half\vrho^n_{i+\half}(u^n_{i+\half})^2.
    \end{equation}
\end{theorem}

\begin{proof}
    Multiply \eqref{eqn:disc-renorm} with $\delx$, sum over $i=1$ to $i=M$ and multiply \eqref{eqn:disc-ke} by $\delx$, sum over $i=0$ to $i=M$, and add the resulting expressions to obtain
    \begin{equation}
    \label{eqn:ent-1}
        \begin{split}
            \sum_{i=1}^M \frac{\delx}{\delt} \bigl(\psig(\vrho^{n+1}_i) - \psig(\vrho^{n}_i)\bigr) &+ \sum_{i=0}^M\frac{\delx}{\delt}\biggl(\half\vrho^{n+1}_{i+\half}(u^{n+1}_{i+\half})^2 - \half\vrho^n_{i+\half}(u^n_{i+\half})^2\biggr) \\
            &\leq -\eta\delt\sum_{i=0}^M \delx (\eth_\E p^{n+1}_{i+\half})^2 - \sum_{i=1}^M \delx R^{n+1}_i - \sum_{i=0}^M \delx S^{n+1}_{i+\half}.
        \end{split}
    \end{equation}
    Here, we have used the conservativity of the mass flux, along with the duality relation \eqref{eqn:disc-grad-div} in order to simplify the resulting expression. Note that $\sum_{i=0}^M\delx R^{n+1}_i\geq 0$ unconditionally because of the convexity of $\psig$. Next, in order to estimate the term $\sum_{n=0}^{M}\delx S^{n+1}_{i+\half}$, we use the identity $(p+q+r)^2\leq 3 (p^2 + q^2 + r^2)$ in the velocity balance \eqref{eqn:disc-vel-update} to obtain 
    \begin{equation}
    \label{eqn:ent-2}
        \begin{split}
            \frac{\vrho^{n+1}_{i+\half}}{2\delt}(u^{n+1}_{i+\half} - u^n_{i+\half})^2 \leq \frac{3\delt}{2\vrho^{n+1}_{i+\half}} \biggl(\frac{1}{\delx^2}(u^n_{i+\frac{3}{2}} &- u^n_{i+\half})^2(F^{n+1, -}_{i+1})^2 \\
            &+ \frac{1}{\delx^2}(u^n_{i-\half} - u^n_{i+\half})^2(F^{n+1, +}_{i})^2 + (\eth_\E p^{n+1}_{i+\half})^2 \biggr).
        \end{split}
    \end{equation}
    Then, we can combine \eqref{eqn:disc-ke-rem} with \eqref{eqn:ent-2} and utilize it in \eqref{eqn:ent-1} in order to yield the following expression:
    \begin{equation}
    \label{eqn:ent-3}
        \begin{split}
            \sum_{i=1}^M \frac{\delx}{\delt} \bigl(\psig(\vrho^{n+1}_i) &- \psig(\vrho^{n}_i)\bigr) + \sum_{i=0}^M\frac{\delx}{\delt}\biggl(\half\vrho^{n+1}_{i+\half}(u^{n+1}_{i+\half})^2 - \half\vrho^n_{i+\half}(u^n_{i+\half})^2\biggr) \\
            &\leq \sum_{i=0}^M \delx\biggl(\frac{F^{n+1,-}_{i+1}}{2\delx}(u^n_{i+\frac{3}{2}} - u^n_{i+\half})^2\biggl(\frac{3}{\vrho^{n+1}_{i+\half}}\frac{\delt}{\delx}F^{n+1,-}_{i+1} + 1\biggr)\biggr) \\
            &+\sum_{i=0}^M \delx\biggl(\frac{F^{n+1,+}_{i}}{2\delx}(u^n_{i+\half} - u^n_{i-\half})^2\biggl(\frac{3}{\vrho^{n+1}_{i+\half}}\frac{\delt}{\delx}F^{n+1,+}_{i} - 1\biggr)\biggr) \\
            &+\delt\sum_{i=0}^M \delx\biggl(\frac{3}{2\vrho^{n+1}_{i+\half}} - \eta\biggr) (\eth_\E p^{n+1}_{i+\half})^2. 
        \end{split}
    \end{equation}
    Now, under the assumptions (1) and (2), the right-hand side of \eqref{eqn:ent-3} remains non-positive. Multiplying by $\delt$ and rearranging the resulting expression gives us the desired inequality \eqref{eqn:disc-loc-ent-ineq}.
\end{proof}

We now state a global energy estimate, which can be proved using the inequality \eqref{eqn:ent-3}.

\begin{theorem}[Global energy estimate]
\label{thm:glob-ent-est}
    Assume that the conditions (1) and (2) stated in Theorem \eqref{thm:loc-ent-ineq} hold. Then, there exists a constant $C>0$ independent of $\veps$ such that for any $1\leq n\leq N$, the following global energy estimate holds:
    \begin{equation}
    \label{eqn:glob-ent-ineq}
        \sum_{i=1}^M\delx\psig(\vrho^n_i) + \sum_{i=0}^M\delx\half\vrho^n_{i+\half}(u^n_{i+\half})^2 +\sum_{r=0}^{n-1}\delt^2\sum_{i=0}^M\delx\biggl(\eta - \frac{3}{2\vrho^{r+1}_{i+\half}}\biggr)(\eth_\E p^{r+1}_{i+\half})^2\leq C.
    \end{equation}
\end{theorem}

\begin{proof}
    As a consequence of \eqref{eqn:ent-3} and the assumptions made, we see that for each $0\leq r\leq N-1$, the following inequality holds
    \begin{equation}
    \label{eqn:glob-ent-1}
        \begin{split}
            \sum_{i=1}^M \frac{\delx}{\delt} \bigl(\psig(\vrho^{r+1}_i) &- \psig(\vrho^{r}_i)\bigr) + \sum_{i=0}^M\frac{\delx}{\delt}\biggl(\half\vrho^{r+1}_{i+\half}(u^{r+1}_{i+\half})^2 - \half\vrho^r_{i+\half}(u^r_{i+\half})^2\biggr) \\
            &+\delt\sum_{i=0}^M \delx\biggl(\eta - \frac{3}{2\vrho^{r+1}_{i+\half}}\biggr) (\eth_\E p^{r+1}_{i+\half})^2\leq 0.
        \end{split}
    \end{equation}
    Multiplying throughout by $\delt$, summing over $r=0$ to $r=n-1$ and utilizing \eqref{eqn:disc-ini-enrg-hyp} yields \eqref{eqn:glob-ent-ineq}.
\end{proof}

The time-step condition required for stability is not only implicit in nature, but it involves the flux as well. Hence, it cannot be implemented in a straightforward manner. However, as done in \cite[Proposition 3.2]{CDV17}, we can derive a sufficient time-step condition which is easier to implement in practice; see also \cite{AGK23, AA24, DVB17, DVB20}. 

\begin{proposition}
\label{prop:suff-tstep}
    Suppose that the time step $\delt>0$ satisfies 
    \begin{equation}
    \label{eqn:suff-tstep}
        \frac{\delt}{\delx}\bigl(\lvert u^n_{i+\half}\rvert + \sqrt{\eta(p^{n+1}_{i+1} - p^{n+1}_i)}\bigr)\leq\max\biggl\lbrace 1, \frac{1}{8}\frac{\min\lbrace \vrho^n_i,\vrho^n_{i+1}\rbrace}{\max\lbrace\vrho^{n+1}_i,\vrho^{n+1}_{i+1}\rbrace}\biggr\rbrace.
    \end{equation}
    Then, $\delt$ also satisfies the condition (1) given in Theorem \ref{thm:loc-ent-ineq}.
\end{proposition}
\begin{proof}
    The proof is exactly the same as detailed in \cite[Proposition 3.2]{CDV17} and we defer to it for the same.
\end{proof}

\begin{remark}
    Note that the sufficient condition \eqref{eqn:suff-tstep} is still implicit in nature. Therefore, we implement it in an explicit manner in our computations.
\end{remark}

\begin{remark}
    It is easy to observe that the condition \eqref{eqn:suff-tstep} implies the time-step $\delt$ satsifies
    \[
    \frac{\delt}{\delx}\bigl(\lvert F^{n+1}_{i+1}\rvert + \lvert F^{n+1}_i\rvert\bigr)\leq \frac{1}{4}\vrho^n_{i+\half}.
    \]
    Hence, we can observe that
    \begin{align*}
        \vrho^{n+1}_{i+\half}-\frac{3}{4}\vrho^n_{i+\half}\geq\vrho^{n+1}_{i+\half}-\vrho^n_{i+\half}+ \frac{\delt}{\delx}\bigl(\lvert F^{n+1}_{i+1}\rvert + \lvert F^{n+1}_i\rvert\bigr)\geq 0.
    \end{align*}
    where the last inequality follows because of the dual mass balance \eqref{eqn:dual-mss-bal}. Therefore, we obtain that 
    \begin{equation}
    \label{eqn:den-ineq}
        \frac{2}{\vrho^n_{i+\half}}\geq\frac{3}{2\vrho^{n+1}_{i+\half}}.
    \end{equation}
    Now, in accordance with the stability analysis, cf.\ Theorem \ref{thm:loc-ent-ineq}, $\eta$ should satisfy $\eta > \frac{3}{2\vrho^{n+1}_{i+\half}}$. Hence, if we choose $\eta>\frac{2}{\vrho^n_{i+\half}}$, the necessary condition will be satisfied and thus, we can also obtain $\eta$ explicitly.

    Furthermore, we can make a theoretical choice for $\eta$ independent of $\veps$ in the following manner, provided the initial density remains bounded below uniformly with respect to $\veps$. From \eqref{eqn:den-ineq}, we can inductively obtain that for each $0\leq n\leq N-1$,
    \[
        2\biggl(\frac{4}{3}\biggr)^n\frac{1}{\vrho^0_{i+\half}}\geq \frac{3}{2\vrho^{n+1}_{i+\half}}.
    \]
    If $\underline{\vrho^0}>0$ is constant independent of $\veps$ such that $\vrho_\epso\geq\underline{\vrho^0}>0$, then we can obtain
    \begin{equation}
    \label{eqn:eta-choice}
        2\biggl(\frac{4}{3}\biggr)^{N-1}\frac{1}{\underline{\vrho^0}}\geq\frac{3}{2\vrho^{n+1}_{i+\half}}.
    \end{equation}
    Therefore, we can choose $\eta$ as 
    \[
    \eta>2\biggl(\frac{4}{3}\biggr)^{N-1}\frac{1}{\underline{\vrho^0}}\geq\frac{3}{2\vrho^{n+1}_{i+\half}},
    \]
    and this this choice ensures that $\eta$ is independent of $\veps$ (but not independent of the discretization) and further satisfies the required stability condition. Therefore, on a fixed mesh, we can choose $\eta$ such that it is a constant independent of $\veps$. In addition, \eqref{eqn:eta-choice} along with the global energy estimate \eqref{eqn:glob-ent-ineq} yields the following $L^2$-estimate which is uniform in $\veps$ for the stabilization term \eqref{eqn:vel-stab-term} for each $0\leq n\leq N-1$
    \begin{equation}
        \label{eqn:stab-term-est}
            \sum_{i=0}^M\delx(\du^{n+1}_{i+\half})^2\leq C\eta^2.
    \end{equation}
\end{remark}

\subsubsection{Density bound and control of pressure}

The upcoming lemma shows that the numerical densities generated by the scheme are strictly bounded above, away from 1. The proof of this result can be directly adapted to the current setting from \cite[Lemma 4.6]{PS22}. The proof utilizes the fact that the free energy $\psig$ is now bounded above uniformly in $\veps$ as a consequence of \eqref{eqn:glob-ent-ineq}. For the sake of brevity, we just state the result and refer to the aforementioned reference for a proof.

\begin{lemma}[Bound on the discrete density]
\label{lem:disc-den-bd}
    Suppose $\gamma\in(1,3\rbrack$. Then, for $\veps$ and $\delx$ sufficiently small, there exists a constant $C_\M > 0$ which is independent of $\veps$ but dependent on the mesh, such that
    \begin{equation}
    \label{eqn:disc-den-bd}
        \vrho^n_i \leq 1-C_\M\veps^{\frac{1}{\gamma - 1}},\quad\text{for each }1\leq i\leq M\text{ and }1\leq n\leq N.
    \end{equation}
\end{lemma}
\begin{proof}
    See \cite[Lemma 4.6]{PS22}.
\end{proof}

\begin{remark}
    The density bound stated in the above lemma is precisely the discrete variant of \eqref{eqn:wk-den-bound}, which is satisfied by weak solutions of system \eqref{eqn:eul-sys}-\eqref{eqn:cong-pres}.
\end{remark}

Given that the numerical densities are strictly less than 1 due to \eqref{eqn:disc-den-bd}, the discrete pressure gradient appearing in the mass balance \eqref{eqn:disc-mss-bal} as part of the mass flux is now well-defined, cf.\ \eqref{eqn:disc-mss-flx} and \eqref{eqn:vel-stab-term}. Further, because of the presence of the stabilization term, the mass balance \eqref{eqn:disc-mss-bal} is non-linear in nature. However, once the updated density $\vrho^{n+1}_i$ is obtained, the momentum balance \eqref{eqn:disc-mom-bal} can be solved explicitly to yield $u^{n+1}_{i+\half}$. As such, we only need to show the existence of $\vrho^{n+1}_i$ from the mass balance, given $(\vrho^n_i, u^n_{i+\half})$. Using the tools of topological degree theory in finite dimensions, cf.\ \cite{Dei85, Nir01}, the existence of a solution to the non-linear mass balance can be shown in an analogous manner as given in \cite{AGK23, GMN19}. For the sake of completeness, we just state the result and defer to \cite[Theorem 4.2]{AGK23}

\begin{theorem}[Existence of a solution to the scheme]
\label{thm:soln-scheme}
    Suppose the given initial state $(\vrho^{n}_i, u^n_{i+\half})$ is such that $\vrho^n_i > 0$ on $\Omega$. Then, there exists a solution $\vrho^{n+1}_i$ of \eqref{eqn:disc-mss-bal} such that $\vrho^{n+1}_i > 0$ on $\Omega$. In particular, if $\vrho^0_i > 0$, then $\vrho^{n}_i > 0$ for each $1\leq n\leq N$. 
\end{theorem}

A control over the pressure term, which is uniform in $\veps$, can now be established using a one dimensional Bogovskii-type estimate as done in \cite{BPZ14}.
\begin{lemma}[Pressure estimate]
\label{lem:pres-est}
    Suppose there exists a constant $\underline{\vrho^0}>0$ independent of $\veps$ such that $\vrho_\epso\geq\underline{\vrho^0}>0$ uniformly for $\veps>0$. Then, for each $0\leq n\leq N-1$, we have 
    \begin{equation}
    \label{eqn:pres-est}
        \sum_{i = 1}^M\delx p^{n+1}_i\leq C_\M,
    \end{equation}
    where $C_\M>0$ is a constant independent of $\veps$.
\end{lemma}

\begin{proof}
    Let $\varphi^{n+1}_{\half} = 0$ and for $1\leq i\leq M$, let $\varphi^{n+1}_{i+\half} = \sum_{j = 1}^{i}\delx(\vrho^{n+1}_j - \overline{\vrho^{n+1}})$, where $\overline{\vrho^{n+1}} = \frac{1}{\abs{\Omega}}\sum_{i = 1}^M\delx\vrho^{n+1}_i$ denotes the average value of $\vrho^{n+1}$. The bound on the density guarantees that $\varphi^{n+1}$ is bounded uniformly with respect to $\veps$. Multiplying the momentum balance \eqref{eqn:disc-mom-bal} with $\delx\varphi^{n+1}_{i+\half}$ and summing over $i$ yields 
    \begin{equation}
    \label{eqn:pres-1}
        \begin{split}
            -\sum_{i=0}^M\delx\varphi^{n+1}_{i+\half}\eth_\E p^{n+1}_{i+\half} &= \sum_{i=0}^M\delx\varphi^{n+1}_{i+\half}\biggl(\frac{1}{\delt}(\vrho^{n+1}_{i+\half}u^{n+1}_{i+\half} - \vrho^n_{i+\half}u^n_{i+\half})\biggr) \\ 
            &\quad+\sum_{i=0}^M\delx\varphi^{n+1}_{i+\half}\biggl(\frac{1}{\delx}(F^{n+1}_{i+1}u^n_{i+1} - F^{n+1}_i u^n_i)\biggr).
        \end{split}
    \end{equation}
    Now, as a consequence of the density bound \eqref{eqn:disc-den-bd} and the energy estimate \eqref{eqn:glob-ent-ineq}, one can obtain a control on the $L^2$-norm of the discrete velocity $u^n$ independent of $\veps$ for each $1\leq n\leq N$. As a consequence, the first term on the right-hand side is bounded uniformly with respect to $\veps$. Further, note that $F^{n+1}_i \leq \half(\tu^n_{i+\half} + \tu^n_{i-\half})$ due to \eqref{eqn:disc-den-bd}. Now, $\tu^n_{i+\half} = u^n_{i+\half} - \du^{n+1}_{i+\half}$ is also bounded in the $L^2$-norm thanks to the estimate \eqref{eqn:stab-term-est} and the aforementioned estimate on the velocity. As a consequence, even the second term on the right-hand side is bounded independent of $\veps$. Finally, using \eqref{eqn:disc-grad-div}, we obtain
    \begin{equation}
    \label{eqn:pres-2}
        \abs{\sum_{i=1}^M\delx p^{n+1}_i\eth_\M\varphi^{n+1}_i} = \abs{\sum_{i=1}^M\delx p^{n+1}_i(\vrho^{n+1}_i - \overline{\vrho^{n+1}})} = \abs{\int_\Omega p^{n+1}(\vrho^{n+1} - \overline{\vrho^{n+1}})\,\dx} \leq C_\M,
    \end{equation}
    where $C_\M>0$ is a constant independent of $\veps$ but dependent on the mesh. Next, the integral on the left-hand side above is split over two domains, $\Omega_1 = \lbrace\vrho^{n+1}\leq \frac{1+\overline{\vrho}}{2}\rbrace$ and $\Omega_2 = \lbrace\vrho^{n+1} > \frac{1+\overline{\vrho}}{2}\rbrace$, where $\overline{\vrho}<1$ is given in \eqref{eqn:ini-data-hyp}. On $\Omega_1$, as the density is far from the singular value 1, the pressure is uniformly bounded with respect to $\veps$ due to the density bound and hence, the corresponding integral is also controlled. Finally, we have
    \begin{equation}
    \label{eqn:pres-3}
    \begin{split}
        C_\M\geq\abs{\int_{\Omega_2}p^{n+1}(\vrho^{n+1} - \overline{\vrho^{n+1}})\,\dx}&\geq \biggl(\frac{1+\overline{\vrho}}{2} - \overline{\varrho^{n+1}}\biggr)\int_{\Omega_2}p^{n+1}\,\dx \\
        &\geq\frac{1-\overline{\varrho}}{2}\int_{\Omega_2}p^{n+1}\,\dx,
    \end{split}
    \end{equation}
    wherein the last inequality follows due to the conservation of mass implied by the mass balance \eqref{eqn:disc-mss-bal}. As $\overline{\varrho}<1$, we obtain that the pressure is uniformly bounded in $L^1(\Omega)$ which completes the proof.
\end{proof}

\subsection{Two-dimensional finite volume scheme}
\label{subsec:2d-scheme}

In two dimensions, the Euler system with the congestion pressure law \eqref{eqn:cong-pres} reads:
\begin{subequations}
\label{eqn:2d-eul-sys}
    \begin{gather}
        \Dt\vrho_\veps + \Dx(\vrho_\veps u_\veps) + \Dy(\vrho_\veps v_\veps) = 0, \label{eqn:2d-mss-bal} \\
        \Dt(\vrho_\veps u_\veps) + \Dx(\vrho_\veps (u_\veps)^2) + \Dy(\vrho_\veps u_\veps v_\veps) + \Dx(p_\veps(\vrho_\veps)) = 0, \label{eqn:2d-x-mom-bal} \\
        \Dt(\vrho_\veps v_\veps) + \Dx(\vrho_\veps u_\veps v_\veps) + \Dy(\vrho_\veps (v_\veps)^2) + \Dy(p_\veps(\vrho_\veps)) = 0, \label{eqn:2d-y-mom-bal} \\
        \vrho_\veps(0,\cdot) = \vrho_\epso,\quad (u_\veps, v_\veps)(0,\cdot) = (u_\epso, v_\epso). \label{eqn:2d-ini-dat}
    \end{gather}
\end{subequations}
for $(t,x)\in\odom$, where $\Omega\subset\R^2$ is an open and bounded set. Here, $\vrho_\veps$ is the density of the fluid and $u_\veps$ and $v_\veps$ denote the $x$ and $y-$components of the fluid velocity respectively. The pressure $p_\veps(\vrho_\veps)$ is given by $\eqref{eqn:cong-pres}$.

Analogous to the one-dimensional case, we introduce the shifted velocities $\tu_\veps = u_\veps - \du_\veps$ and $\tv_\veps = v_\veps - \delta v_\veps$. Here, $\du_\veps = \eta\Dx p_\veps(\vrho_\veps)$ and $\delta v_\veps = \eta \Dy p_\veps(\vrho_\veps)$ where $\eta>0$. We consider the following velocity stabilized system:
\begin{subequations}
\label{eqn:2d-eul-sys-stab}
    \begin{gather}
        \Dt\vrho_\veps + \Dx(\vrho_\veps \tu_\veps) + \Dy(\vrho_\veps \tv_\veps) = 0, \label{eqn:2d-mss-bal-stab} \\
        \Dt(\vrho_\veps u_\veps) + \Dx(\vrho_\veps \tu_\veps u_\veps) + \Dy(\vrho_\veps \tu_\veps v_\veps) + \Dx(p_\veps(\vrho_\veps)) = 0, \label{eqn:2d-x-mom-bal-stab} \\
        \Dt(\vrho_\veps v_\veps) + \Dx(\vrho_\veps \tv_\veps u_\veps) + \Dy(\vrho_\veps \tv_\veps v_\veps) + \Dy(p_\veps(\vrho_\veps)) = 0. \label{eqn:2d-y-mom-bal-stab}
    \end{gather}
\end{subequations}

As the existence of weak solutions to \eqref{eqn:2d-eul-sys} is unknown, we simply design a finite volume scheme in order to simulate the solutions without performing the subsequent analysis as done in the one-dimensional case. 

For the simplicity of exposition, let $\Omega = (a,b)\times (c,d)$. As before, let $0=t_0<t^1<\cdots<t^n<\cdots<t^N = T$ be a time discretization of $(0,T)$ with $\delt = t^{n+1} - t^n$ being the constant time-step. We consider a two-dimensional MAC discretization $(\M,\E)$ of $\Omega$, wherein $\M$ denotes the collection of all primal cells of the form $K_{i,j} = \lbrack x_{i-\half}, x_{i+\half}\rbrack\times\lbrack y_{j-\half}, y_{j+\half}\rbrack$ such that $\overline{\Omega} = \bigcup_{i,j} K_{i,j}$. Denote by $\E = \E^{(1)}\cup \E^{(2)}$, the collection of all edges $\sigma_{i+\half,j} = \lbrace x_{i+\half}\rbrace\times\lbrack y_{j-\half}, y_{j+\half}\rbrack\in\E^{(1)}$ and $\sigma_{i,j + \half} = \lbrack x_{i-\half}, x_{i+\half}\rbrack\times\lbrace y_{j+\half}\rbrace\in\E^{(2)}$. Note that $\E^{(k)}$ is the collection of all edges perpendicular to the $k$-th standard basis vector $\mathbf{e}^{(k)}$, $k = 1,2$. The cell center of $K_{i,j}$ is denoted by $(x_i, y_j)$, where $x_i = \half(x_{i-\half} + x_{i+\half})$ and $y_j = \half(y_{j-\half} + y_{j+\half})$. To each edge $\sigma_{i+\half,j}\in\E^{(1)}$, we associate a dual cell $D_{i+\half,j} = \lbrack x_{i}, x_{i+1}\rbrack\times\lbrack y_{j-\half}, y_{j+\half}\rbrack$ and analogously, for $\sigma_{i,j+\half}\in\E^{(2)}$, we set $D_{i,j+\half} = \lbrack x_{i-\half}, x_{i+\half}\rbrack\times \lbrack y_{j}, y_{j+1}\rbrack$. For the sake of simplicity, we suppose that the space steps $\delx = x_{i+\half} - x_{i-\half} = x_{i+1} - x_i$ and $\dely = y_{j+\half} - y_{j-\half} = y_{j+1} - y_j$ are constants. Because of the MAC discretization, the density and the pressure are approximated at the cell centres of the primal cells $K_{i,j}$, while the velocities are approximated at the centres of the dual cells. In particular, the $x$-component of the velocity, $u_\veps$, is approximated on the dual cells $D_{i+\half, j}$ associated to $\sigma_{i+\half,j}$ and the $y$-component of the velocity, $v_\veps$, is approximated on the dual cells $D_{i,j+\half}$ associated to $\sigma_{i,j+\half}\in\E^{(2)}$. The unknowns are thus the density $\vrho_{i,j}$, the $x$-component of the velocity $u_{i+\half,j}$, and the $y$-component of the velocity $v_{i, j+\half}$. We refer to Figure \ref{fig:2d-mac-disc} for a pictorial representation of the grid as well as the arrangement of the variables for a two-dimensional MAC discretization.

We initialize the scheme by setting 
\begin{gather*}
\label{eqn:schm-init-2d}
    \vrho^0_{i,j} = \frac{1}{\delx\dely}\int_{K_{i,j}}\vrho_\epso \,\dx\,\mathrm{d}y,\\
    u^0_{i+\half,j} = \frac{1}{\delx\dely}\int_{D_{i+\half,j}} u_\epso \,\dx\,\mathrm{d}y,
    \quad v^0_{i,j+\half} = \frac{1}{\delx\dely}\int_{D_{i,j+\half}}v_\epso\,\dx\,\mathrm{d}y.
\end{gather*}

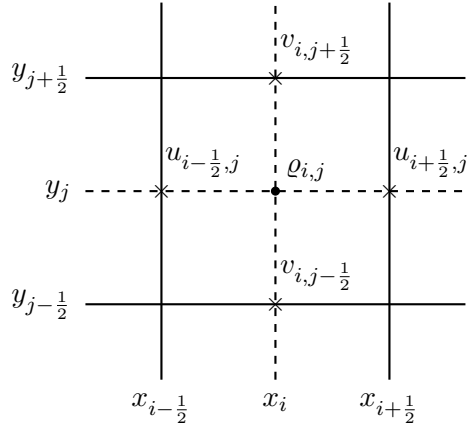
\begin{figure}
	\centering
	\begin{tikzpicture}[scale = 0.5]
		\draw [thick](1,0) -- (11,0);
		\draw [black,thick,style=dashed](1,3) -- (11,3);
		\draw [thick](1,6) -- (11,6);
		\draw [thick](3,-2) -- (3,8);
		\draw [black,thick,style=dashed](6,-2) -- (6,8);
		\draw [thick](9,-2) -- (9,8);
		\draw[black] (0.9,0)  node[anchor=east]{$y_{j-\half}$};
		\draw[black] (0.9,3)  node[anchor=east]{$y_j$};
		\draw[black] (0.9,6)  node[anchor=east]{$y_{j+\half}$};
		\draw[black] (3,-2.1)  node[anchor=north]{$x_{i-\half}$};
		\draw[black] (6,-2.1)  node[anchor=north]{$x_{i}$};
		\draw[black] (9,-2.1)  node[anchor=north]{$x_{i+\half}$};
		\draw[black] (3.5,-0.2) node[anchor=north]{} ;
		\draw[black] (9.8,-0.2) node[anchor=north]{} ;
		\draw[black] (3.8,5.8) node[anchor=north]{} ;
		\draw[black] (10.1,5.8) node[anchor=north]{} ;
		\draw[black] (6,6) node {\small$\times$};
		\draw[black] (7.1,6.1) node[anchor=south] {$v_{i,j+\half}$};
            \filldraw[black] (6,3) circle (3pt) node[anchor=south west]{$\varrho_{i,j}$};
		\draw[black] (6,0) node {\small$\times$};
		\draw[black] (6.9,0.1) node[anchor=south] {\hspace{0.2cm}$v_{i,j-\half}$};
		\draw[black] (3,3) node {\small$\times$};
		\draw[black] (3.9,3.1) node[anchor=south] {\hspace{0.2cm}$u_{i-\half,j}$};
		\draw[black] (9,3) node {\small$\times$};
		\draw[black] (10.1,3.1) node[anchor=south] {$u_{i+\half,j}$};
	\end{tikzpicture}
	\caption{Arrangement of the variables for a two-dimensional MAC discretization.}
        \label{fig:2d-mac-disc}
\end{figure}

As done in the one-dimensional case, we introduce a semi-implicit in time finite volume scheme to approximate \eqref{eqn:2d-eul-sys-stab}, and the scheme reads:

\begin{subequations}
\label{eqn:2d-scheme}
    \begin{gather}
            \frac{1}{\delt}(\vrho^{n+1}_{i,j} - \vrho^n_{i,j}) + \frac{1}{\delx}(F^{n+1}_{i+\half, j} - F^{n+1}_{i-\half, j})
            + \frac{1}{\dely}(G^{n+1}_{i,j+\half} - G^{n+1}_{i,j-\half}) = 0, \label{eqn:2d-disc-mss-bal} \\
            \frac{1}{\delt}(\vrho^{n+1}_{i+\half,j} u^{n+1}_{i+\half,j} - \vrho^{n}_{i+\half,j}u^n_{i+\half, j}) + \frac{1}{\delx}(F^{n+1}_{i+1,j}u^n_{i+1,j} - F^{n+1}_{i,j}u^n_{i,j}) \nonumber \\
            +\frac{1}{\dely}(G^{n+1}_{i+\half,j+\half}u^n_{i+\half,j+\half} - G^{n+1}_{i+\half,j-\half}u^n_{i+\half,j-\half}) +\eth_x p^{n+1}_{i+\half, j} = 0, \label{eqn:2d-disc-x-mom-bal} \\
            \frac{1}{\delt}(\vrho^{n+1}_{i,j+\half} v^{n+1}_{i,j+\half} - \vrho^{n}_{i,j+\half}v^n_{i,j+\half}) +\frac{1}{\delx}(F^{n+1}_{i+\half,j+\half}v^n_{i+\half,j+\half} - F^{n+1}_{i-\half,j+\half}v^n_{i-\half,j+\half})\nonumber \\
            \frac{1}{\dely}(G^{n+1}_{i,j+1}v^n_{i,j+1} - G^{n+1}_{i,j}v^n_{i,j}) +\eth_y p^{n+1}_{i, j+\half} = 0. 
            \label{eqn:2d-disc-y-mom-bal}
    \end{gather}
\end{subequations}

The mass fluxes are defined as 
\begin{subequations}
\label{eqn:2d-mss-flux}
    \begin{align}
        &F^{n+1}_{i+\half,j} = \vrho^{n+1}_{i,j}\tu^{n+1,+}_{i+\half,j} + \vrho^{n+1}_{i+1,j}\tu^{n+1,-}_{i+\half,j}, \\
        &G^{n+1}_{i,j+\half} = \vrho^{n+1}_{i,j}\tv^{n+1,+}_{i,j+\half} + \vrho^{n+1}_{i,j+1}\tv^{n+1,-}_{i,j+\half},
    \end{align}
\end{subequations}
where $\tu^{n+1}_{i+\half,j} = u^n_{i+\half,j} - \du^{n+1}_{i+\half,j}$, $\tv^{n+1}_{i,j+\half} = v^n_{i,j+\half} - \delta v^{n+1}_{i,j+\half}$. Analogous to the one-dimensional case, we set 
\begin{equation}
    \du^{n+1}_{i+\half,j} = \eta\delt\eth_x p^{n+1}_{i+\half,j}, \quad \delta v^{n+1}_{i,j+\half} = \eta\delt\eth_y p^{n+1}_{i,j+\half}.
\end{equation}

The positive and negative halves of the shifted velocities appearing in the mass flux are defined as 
\begin{subequations}
\label{eqn:vel-split-2d}
    \begin{align}
        &\tu^{n+1,+}_{i+\half,j} = \max\lbrace u^n_{i+\half,j}, 0\rbrace - \min\lbrace\du^{n+1}_{i+\half,j},0\rbrace, \quad
        \tu^{n+1,-}_{i+\half,j} = \min\lbrace u^n_{i+\half,j}, 0\rbrace - \max\lbrace\du^{n+1}_{i+\half,j},0\rbrace, \\
        &\tv^{n+1,+}_{i,j+\half} = \max\lbrace v^n_{i,j+\half}, 0\rbrace - \min\lbrace\delta v^{n+1}_{i,j+\half},0\rbrace, \quad
        \tv^{n+1,-}_{i,j+\half} = \min\lbrace v^n_{i,j+\half}, 0\rbrace - \max\lbrace\delta v^{n+1}_{i,j+\half},0\rbrace.
    \end{align}    
\end{subequations}

The fluxes and the upwind velocities appearing in the $x$-momentum balance \eqref{eqn:2d-disc-x-mom-bal} are defined as 
\begin{subequations}
\begin{align}
    &F^{n+1}_{i,j} = \half(F^{n+1}_{i+\half,j} + F^{n+1}_{i-\half,j}), \quad G^{n+1}_{i+\half,j+\half} = \half(G^{n+1}_{i,j+\half} + G^{n+1}_{i+1,j+\half}), \\
    &u^n_{i,j} = 
    \begin{dcases}
        u^n_{i+\half,j},& \text{if }F^{n+1}_{i,j}\geq 0, \\
        u^n_{i-\half,j},& \text{otherwise}.
    \end{dcases}
    ,\quad u^n_{i+\half,j+\half} = 
    \begin{dcases}
        u^n_{i+\half,j},& \text{if }G^{n+1}_{i+\half,j+\half}\geq 0, \\
        u^n_{i+\half,j+1},& \text{otherwise}.
    \end{dcases}
\end{align}
\end{subequations}

Analogously, the fluxes and upwind velocities appearing the the $y$-momentum balance \eqref{eqn:2d-disc-y-mom-bal} are defined as
\begin{subequations}
\begin{align}
    &G^{n+1}_{i,j} = \half(G^{n+1}_{i,j+\half} + G^{n+1}_{i,j-\half}), \quad F^{n+1}_{i+\half,j+\half} = \half(F^{n+1}_{i+\half,j} + F^{n+1}_{i+\half,j+1}), \\
    &v^n_{i,j} = 
    \begin{dcases}
        v^n_{i,j+\half},& \text{if }G^{n+1}_{i,j}\geq 0, \\
        v^n_{i,j-\half},& \text{otherwise}.
    \end{dcases}
    ,\quad v^n_{i+\half,j+\half} = 
    \begin{dcases}
        v^n_{i,j+\half},& \text{if }F^{n+1}_{i+\half,j+\half}\geq 0, \\
        v^n_{i+1,j+\half},& \text{otherwise}.
    \end{dcases}
\end{align}
\end{subequations}

The interface densities $\vrho^{n}_{i+\half,j},\vrho^{n}_{i,j+\half} $ appearing in \eqref{eqn:2d-disc-x-mom-bal} and \eqref{eqn:2d-disc-y-mom-bal} respectively are defined as 
\begin{equation}
    \vrho^n_{i+\half,j} = \half(\vrho^n_{i,j} + \vrho^n_{i+1,j}),\quad \vrho^n_{i,j+\half} = \half (\vrho^n_{i,j} + \vrho^n_{i,j+1}).
\end{equation}

Finally, the discrete pressure derivative terms read
\begin{equation}
    \eth_x p^{n+1}_{i+\half,j} = \frac{1}{\delx}(p^{n+1}_{i+1,j} - p^{n+1}_{i,j}),\quad \eth_y p^{n+1}_{i, j+\half} = \frac{1}{\dely}(p^{n+1}_{i,j+1} - p^{n+1}_{i,j}).
\end{equation}

\begin{remark}
    In order to make sure that the scheme is stable, we choose the time step $\delt$ and the parameter $\eta$ for the two-dimensional scheme in accordance with the choices presented in \cite{AGK23}.
\end{remark}

\section{Numerical Results}
\label{sec:num-res}

In the experiments that follow, we only consider periodic boundary conditions and set $\gamma = 2$, cf.\ \eqref{eqn:cong-pres}, for each case study. Further, we let $q=\vrho u$ denote the momentum in one-dimension and let $q^x = \vrho u$, $q^y = \vrho v$ denote the $x,y$-components of the momentum in two-dimensions. For one-dimensional case studies, $M$ will denote the number of mesh points and for two-dimensional problems, we consider an $M_x\times M_y$ grid, where $M_x$ denotes the number of cells in the $x$-direction and $M_y$ denotes the number of cells in the $y-$direction. In addition, we consider the mesh to be uniform.

\subsection*{Example 1}
\label{subsec:eg-1}
We consider the following Riemann problem from \cite{DHN11} on the domain $\lbrack 0,1\rbrack$:
\begin{equation}
\label{eqn:eg-1}
    (\vrho_\epso(x), q_\epso(x)) = 
    \begin{dcases}
    (0.7, 0.8),& \text{if }x\in\lbrack 0,0.5), \\
    (0.7,-0.8),& \text{if }x\in(0.5,1\rbrack. 
    \end{dcases} 
\end{equation}
The main aim of this experiment is to showcase the ability of the scheme to preserve the maximal density bound and capture the formation of congestion. The exact solution emanating from the given initial data corresponds to two symmetric shocks moving in opposite directions with a congested intermediate state. We set $M = 200$ and run the simulation for a final time of $t = 0.05$ and present the results in Figure \ref{fig:p1} for $\veps = 10^{-2}, 10^{-4}$. As we can see, the scheme is able to capture the propagation of shockwave quite accurately, given the correspondence with the exact solution in both the cases of $\veps = 10^{-2}, 10^{-4}$.

\begin{figure}[htpb]
    \centering
    \begin{subfigure}{0.3\textheight}
        \includegraphics[width = \textwidth]{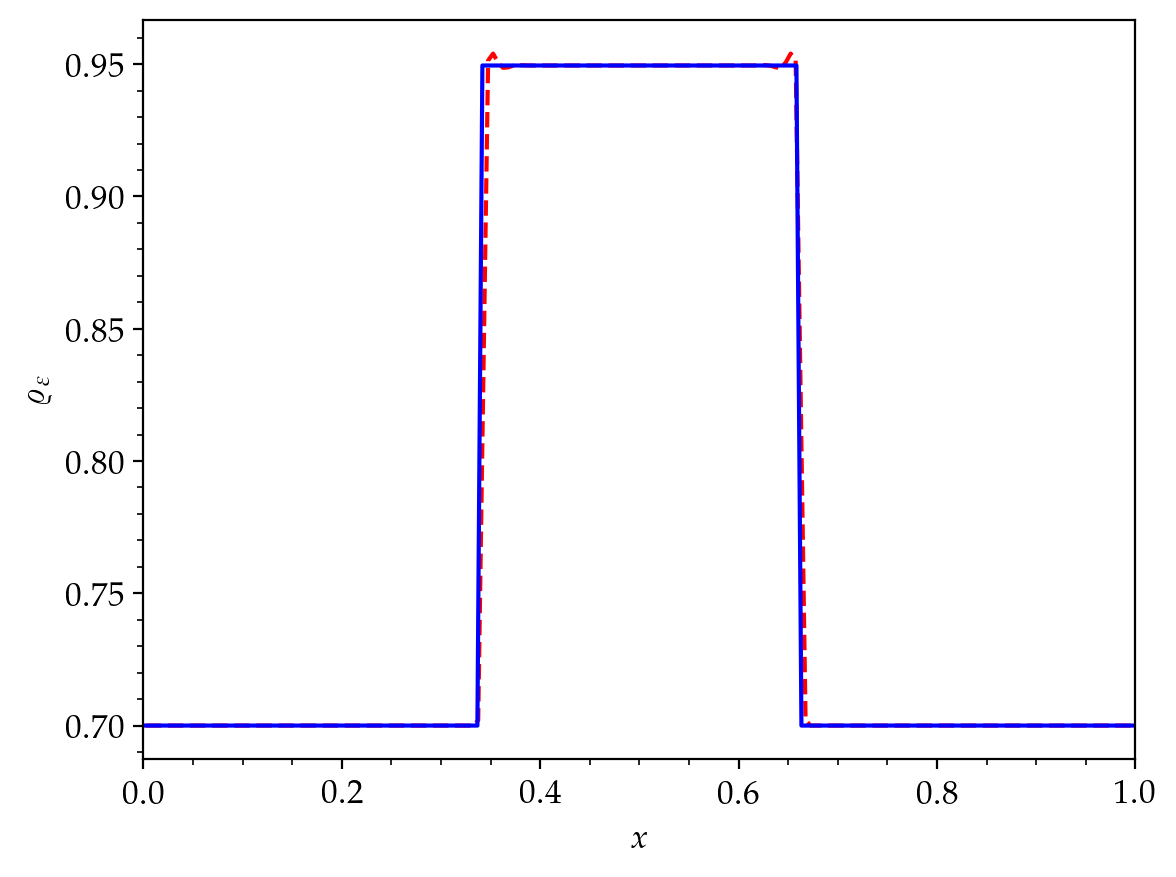}
        \includegraphics[width = \textwidth]{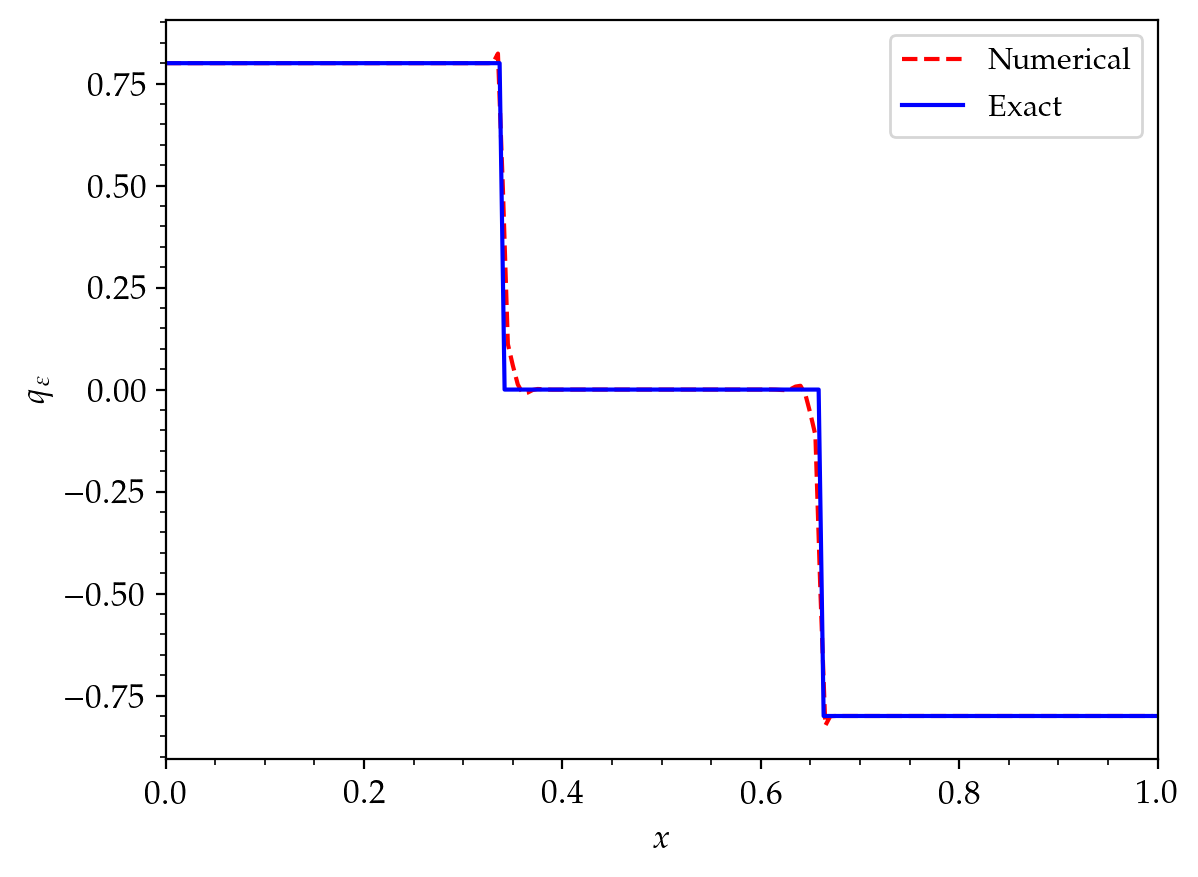}
        \caption{$\veps = 10^{-2}$}
        \label{subfig:p1-eps-2}    
    \end{subfigure}
    \begin{subfigure}{0.3\textheight}
        \includegraphics[width = \textwidth]{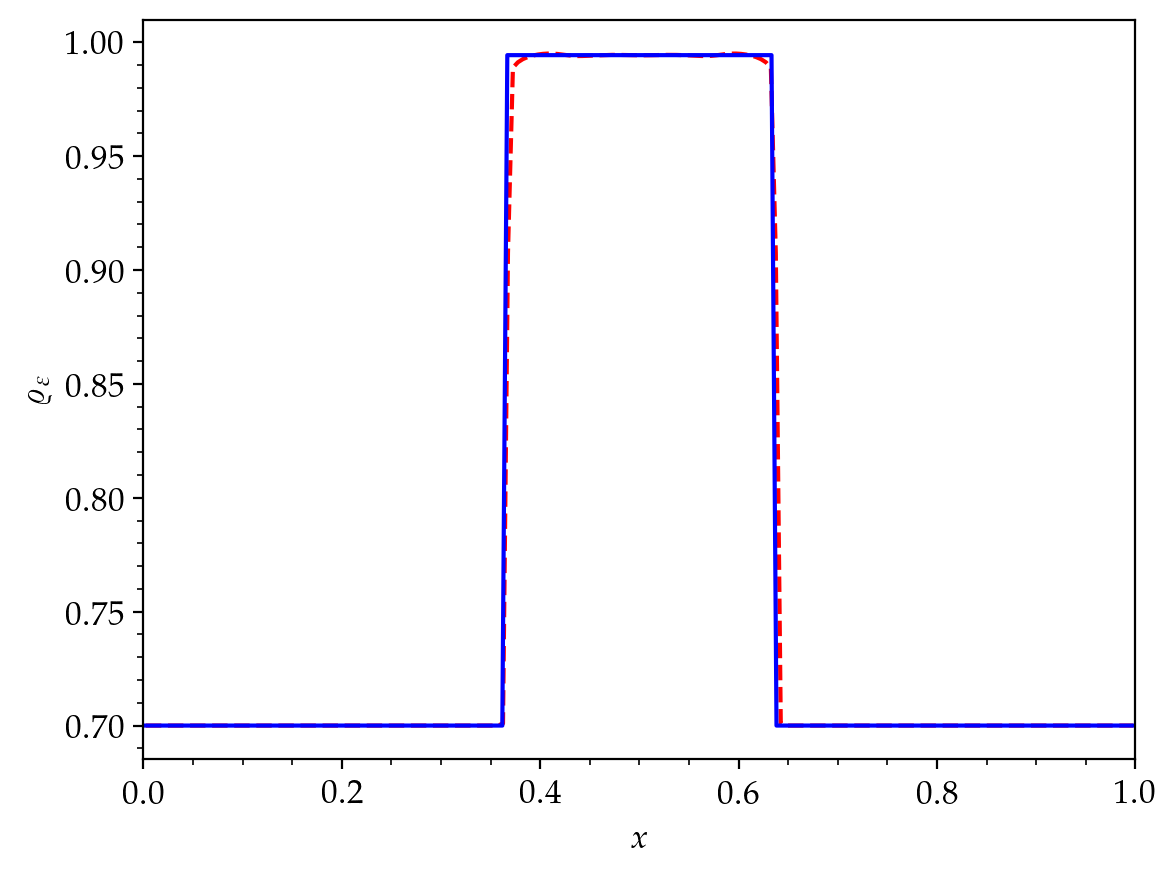}
        \includegraphics[width = \textwidth]{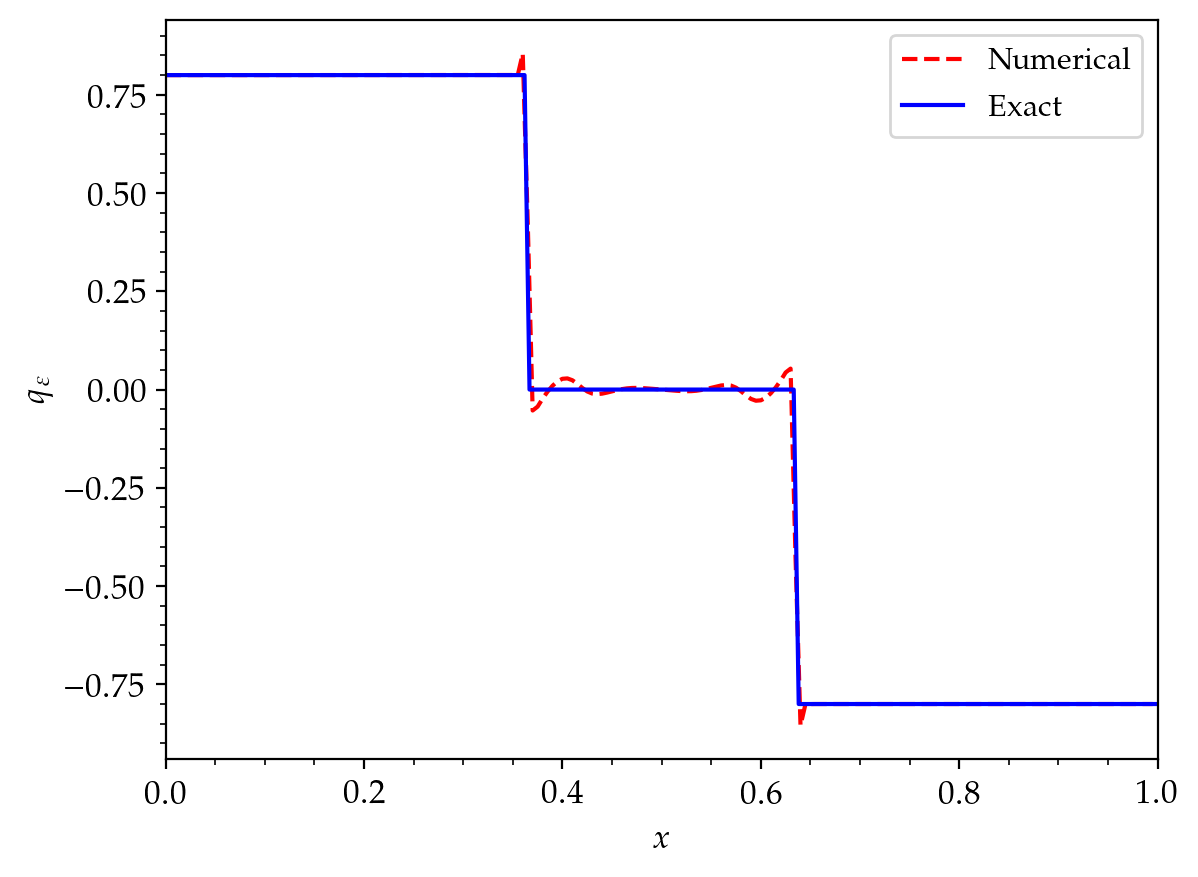}
        \caption{$\veps = 10^{-4}$}
        \label{subfig:p1-eps-4}    
    \end{subfigure}
    \caption{Comparison of the numerical solutions versus the exact solutions with initial data \eqref{eqn:eg-1} for different values of $\veps$. Density (top) and momentum (bottom).}
    \label{fig:p1}
\end{figure}

\subsection*{Example 2}
\label{subsec:eg-2}
We once again consider a Riemann problem from \cite{DHN11} on the domain $\lbrack 0,1\rbrack$. The goal of this problem is to illustrate the positivity preserving property of the present scheme. The initial data reads:
\begin{equation}
\label{eqn:eg-2}
    (\vrho_\epso(x), q_\epso(x)) = 
    \begin{dcases}
    (0.7, -0.8),& \text{if }x\in\lbrack 0,0.5), \\
    (0.7, 0.8),& \text{if }x\in(0.5,1\rbrack. 
    \end{dcases} 
\end{equation}
The exact solution emanating from the given initial data corresponds to two symmetric rarefaction waves moving away from each other, which gives rise to the formation of a vacuum state in between. We set $M = 200$ and run the simulation for a final time of $t = 0.05$ and present the results in Figure \ref{fig:p2} for $\veps = 10^{-4}$. As we can see, the vacuum state is approximated quite well and further, we did not observe any instability during the computations, thus verifying that the scheme is indeed positivity preserving.

\begin{figure}[htpb]
    \centering
    \includegraphics[width = 0.4\textwidth]{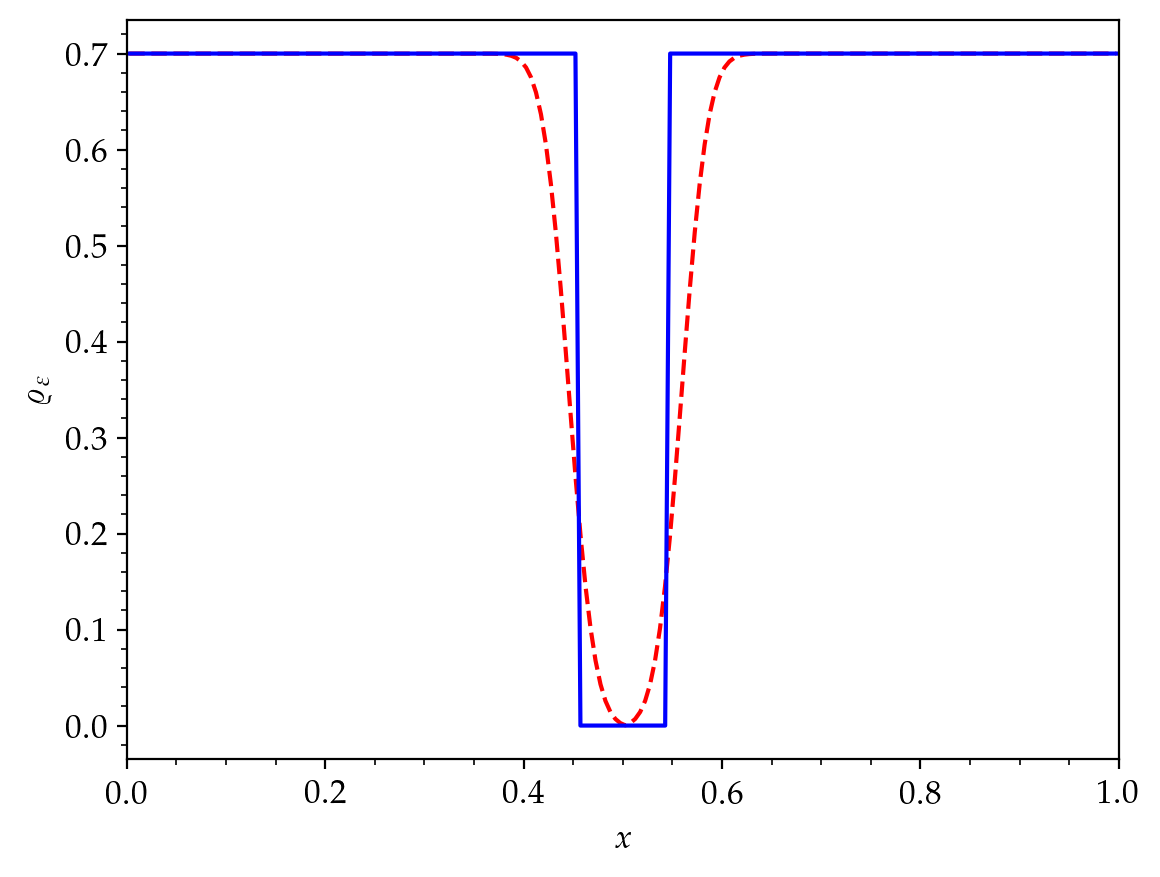}
    \includegraphics[width = 0.4\textwidth]{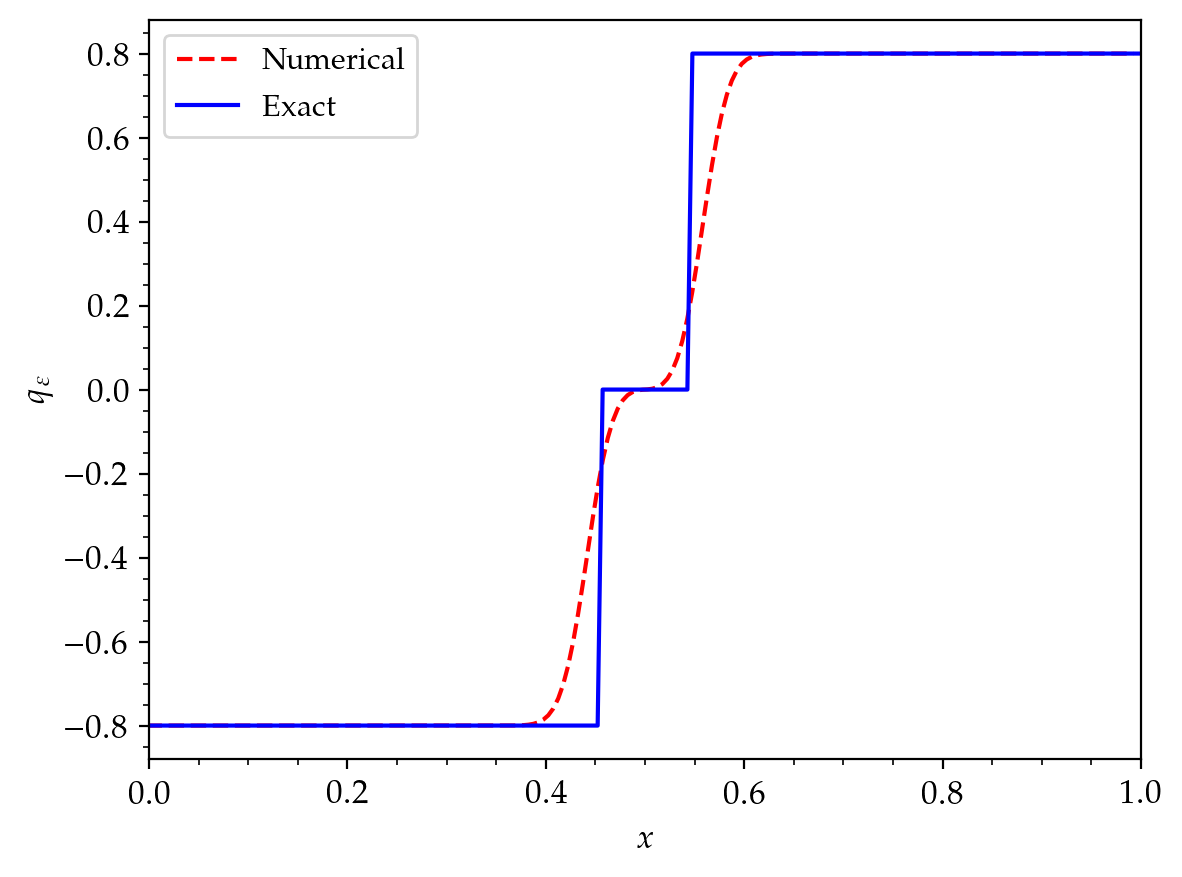}
    \label{fig:p2-eps-4}    
    \caption{Comparison of the numerical solutions versus the exact solutions with initial data \eqref{eqn:eg-2} for $\veps = 10^{-4}$. Density (left) and momentum (right).}
    \label{fig:p2}
\end{figure}

\subsection*{Example 3}
\label{subsec:eg-3}

We consider the following Riemann initial data from \cite{YWS13} for the presureless Euler system, for which the exact solution is available. The goal of this experiment is to show that as $\veps\to 0$, the numerical solutions converge to the exact solution of the presureless system and thus, verifying the AP property of the scheme. Further, we also want to ensure that the scheme is able to maintain positivity of density in the limit $\veps\to 0$.

We consider the domain $\lbrack -0.5, 0.5\rbrack$ and the initial data reads 
\begin{equation}
\label{eqn:eg-3}
    \vrho_\epso(x) = 0.5,\quad u_\epso(x) = \begin{dcases}
        -0.5,& \text{if }x<0, \\
        0.4,& \text{if }x>0.
    \end{dcases}
\end{equation}

The exact solution for the presureless system emanating from the above initial data is given by 
\begin{equation}
\label{eqn:ex-eg-3}
    (\vrho(t,x), u(t,x)) = \begin{dcases}
        (0.5, -0.5),& \text{if }x < -0.5t, \\
        (0,\text{undefined}),& \text{if } -0.5t < x < 0.4t, \\
        (0.5, 0.4),& \text{if }x > 0.4t.
    \end{dcases}
\end{equation}

We set $t = 0.2$, $M = 200$ and consider $\veps = 10^{-k}, k = 4,5,6,7$. From Figure \ref{fig:pless-posi}, we can observe that the scheme is able to approximate the vacuum state even when $\veps\sim 0$, given the correspondence of the numerical solutions versus the exact solution. Furthermore, the pressure profiles presented in Figure \ref{fig:pless-pres} allow us to conclude that the numerical solutions are indeed converging towards a solution of the pressureless Euler system.

\begin{figure}[htpb]
    \centering
    \includegraphics[width = 0.4\textwidth]{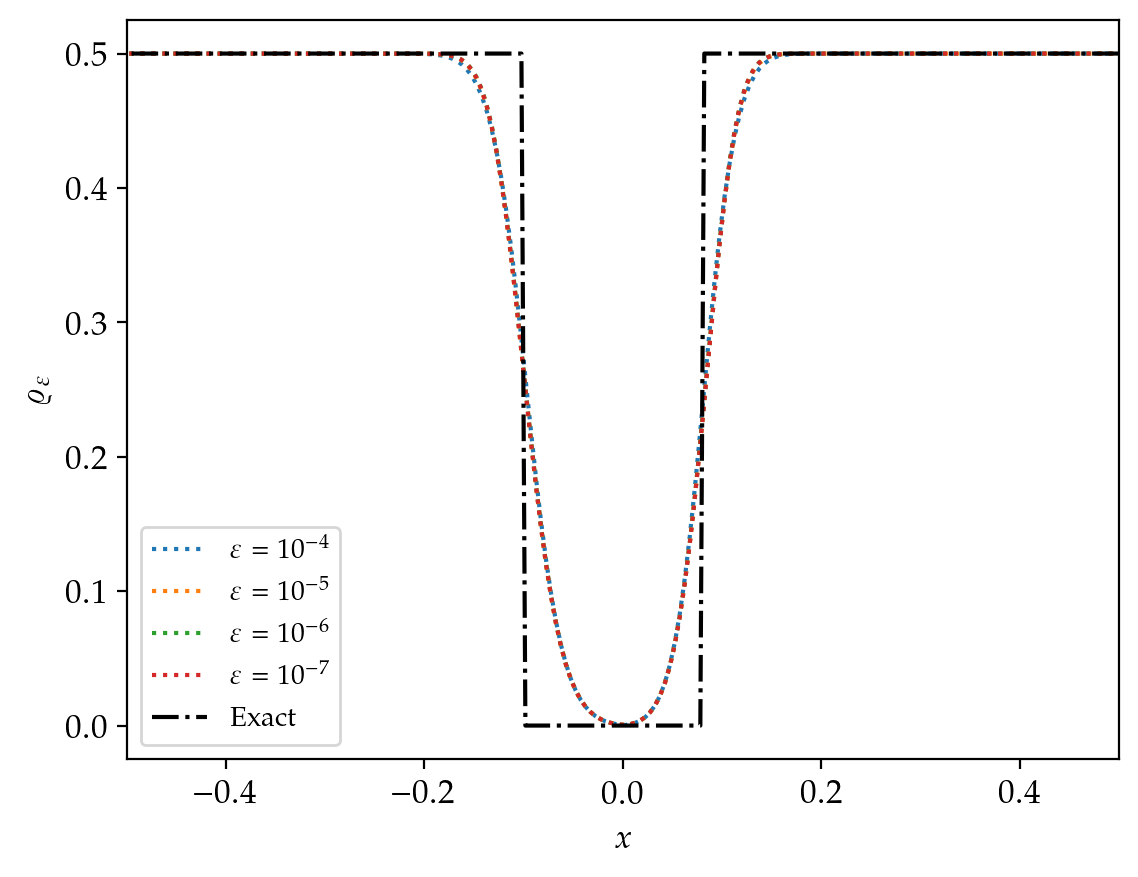}
    \includegraphics[width = 0.4\textwidth]{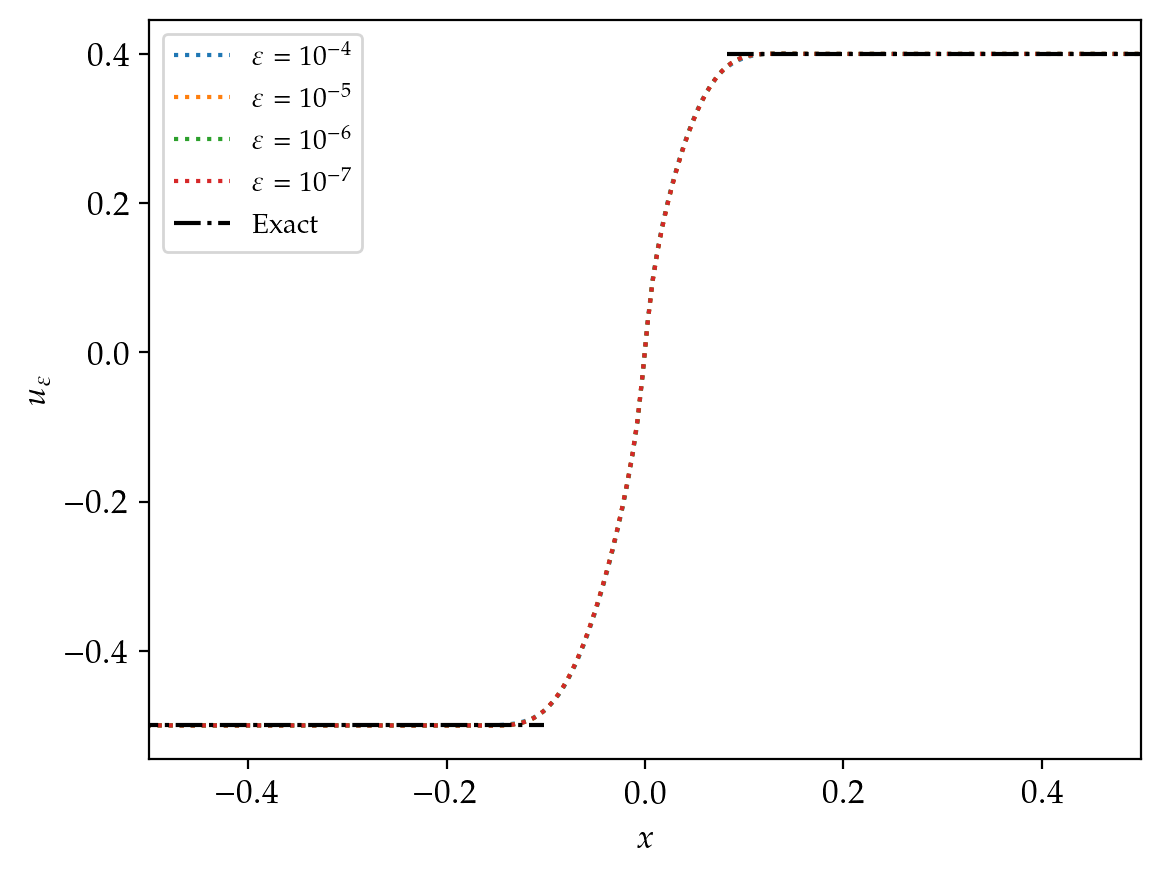}
    \caption{Example 3: density(left) and velocity(right) for different values of $\veps$.}
    \label{fig:pless-posi}
\end{figure}

\begin{figure}[htpb]
    \centering
    \includegraphics[width = 0.5\textwidth]{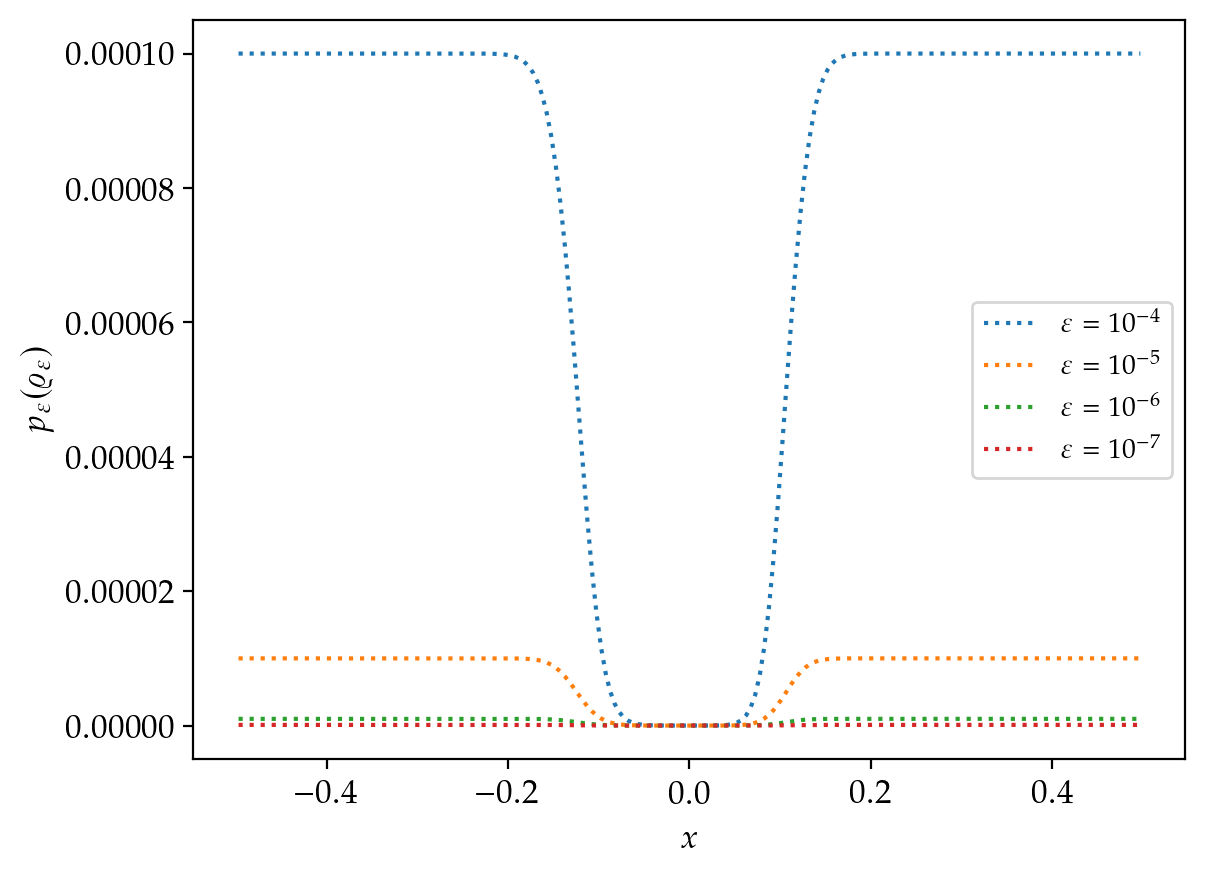}
    \caption{Example 3: the pressure $p_\veps(\vrho_\veps)$ for different values of $\veps$.}
    \label{fig:pless-pres}
\end{figure}

\subsection*{Example 4}
\label{subsec:eg-4}

We consider the following initial data on $\lbrack -1,1\rbrack$ from \cite{YWS13} for the presureless Euler system. The initial data is as follows.
\begin{equation}
\label{eqn:eg-4}
    \vrho_\epso(x) = 0.5,\quad
    u_\epso(x) = \begin{dcases}
        -0.5,& \text{if } x<-0.5, \\
        0.4,& \text{if }-0.5<x<0, \\
        0.4-x,& \text{if }0<x<0.8, \\
        -0.4,& \text{if }x>0.8.
    \end{dcases}
\end{equation}
The above initial data gives rise to a solution of the presureless system that involves the formation of vacuum, discontinuities as well as the formation of congested regions where density is maximal. As such, the goal of this problem is two-fold, first to ensure that the scheme is able to capture such complex configurations involving vacuum, discontinuities and congestion, and further ensure that the scheme is able to approximate the solutions appropriately as $\veps\to 0$ in order to verify the AP property of the scheme.

The explicit formula for the exact solution corresponding to the above initial data \eqref{eqn:eg-4} for $t < 1$ is given by 
\begin{equation}
\label{eqn:ex-eg-4}
        (\vrho(t,x), u(t,x)) = 
        \begin{dcases}
            (0.5, -0.5),& \text{if }x < -0.5-0.5t, \\
            (0,\text{undefined}),& \text{if }-0.5-0.5t < x <-0.5+0.4t, \\
            (0.5, 0.4),& \text{if }-0.5+0.4t < x < 0.4t, \\
            \biggl(\frac{0.5}{1-t},\frac{0.4-x}{1-t}\biggr),& \text{if } 0.4t < x < 0.8-0.4t, \\
            (0.5, -0.4),& \text{if }x > 0.8-0.4t.
        \end{dcases}
\end{equation}

If $t<0.5$, $\vrho(t,x)<1$ for any $x\in\lbrack -1,1\rbrack$. Hence, we run our computations for a final time of $t = 0.49$ and consider $\veps = 10^{-k}, k = 4,5,6,7$ and $M = 200$. From Figure \ref{fig:yang-den-vel}, we can clearly see that as $\veps$ decreases, we observe an extremely good correspondence between the numerical density and velocity and their respective exact solutions. In addition, we see that the positivity of the density is still maintained despite $\veps\sim 0$, showcasing the schemes ability to maintain positivity of density even in extreme conditions. From Figure \ref{fig:pres_yang}, we note that the pressure $p_\veps(\vrho_\veps)$ is only activated in the region where the density is nearly 1. Furthermore, it is only of order $10^{-3}$ and is decreasing to 0 as $\veps$ is decreasing, which clearly indicates the convergence of the numerical solutions to the solution of the presureless system.
 
\begin{figure}[htpb]
    \centering
    \includegraphics[width = 0.4\textwidth]{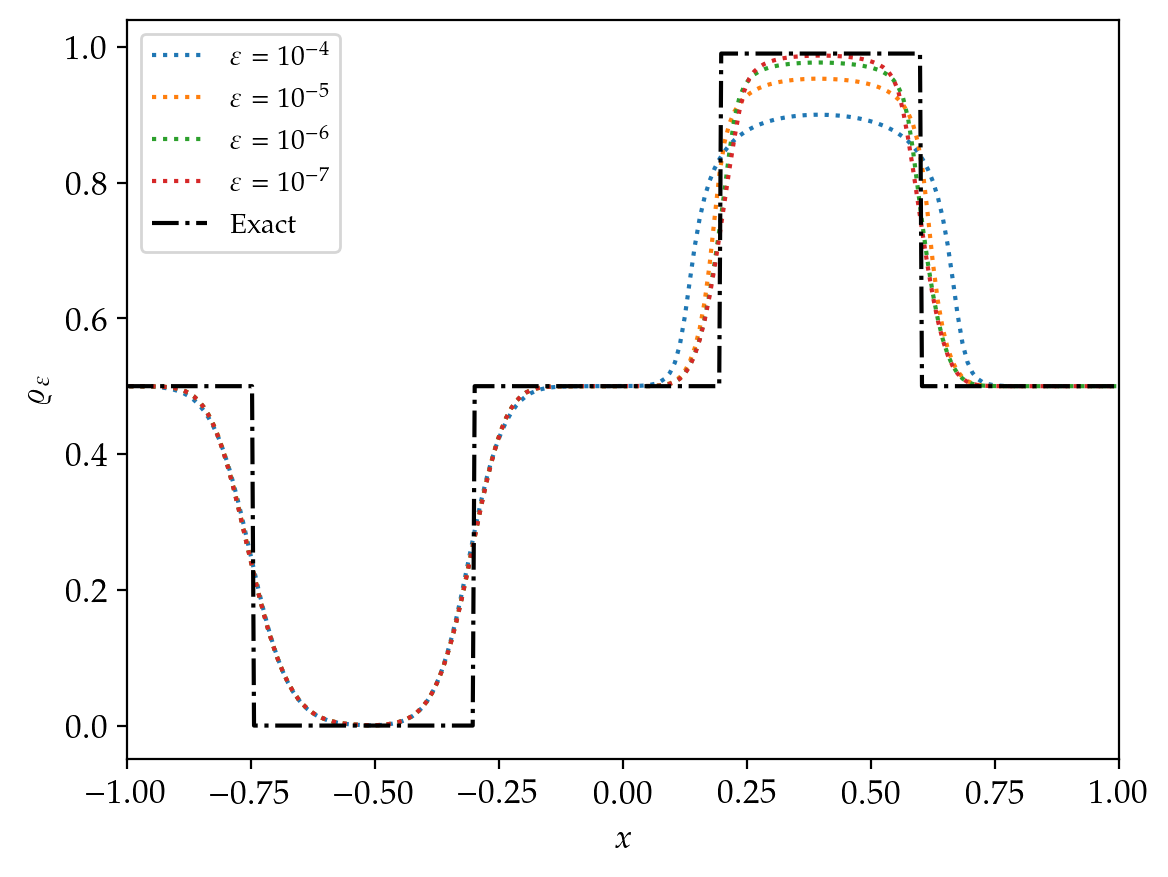}
    \includegraphics[width = 0.4\textwidth]{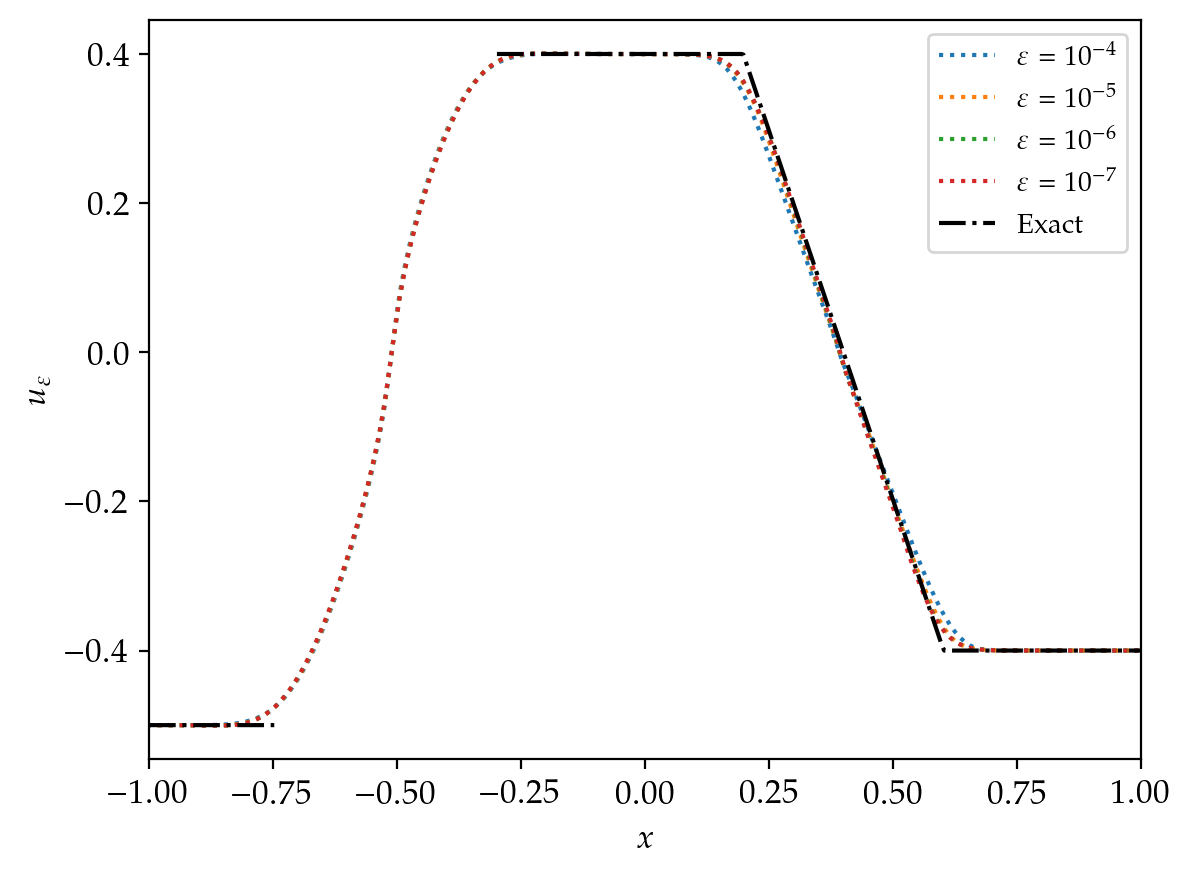}
    \caption{Example 4: density(left) and velocity(right) for different values of $\veps$.}
    \label{fig:yang-den-vel}
\end{figure}

\begin{figure}[htpb]
    \centering
    \includegraphics[width = 0.5\textwidth]{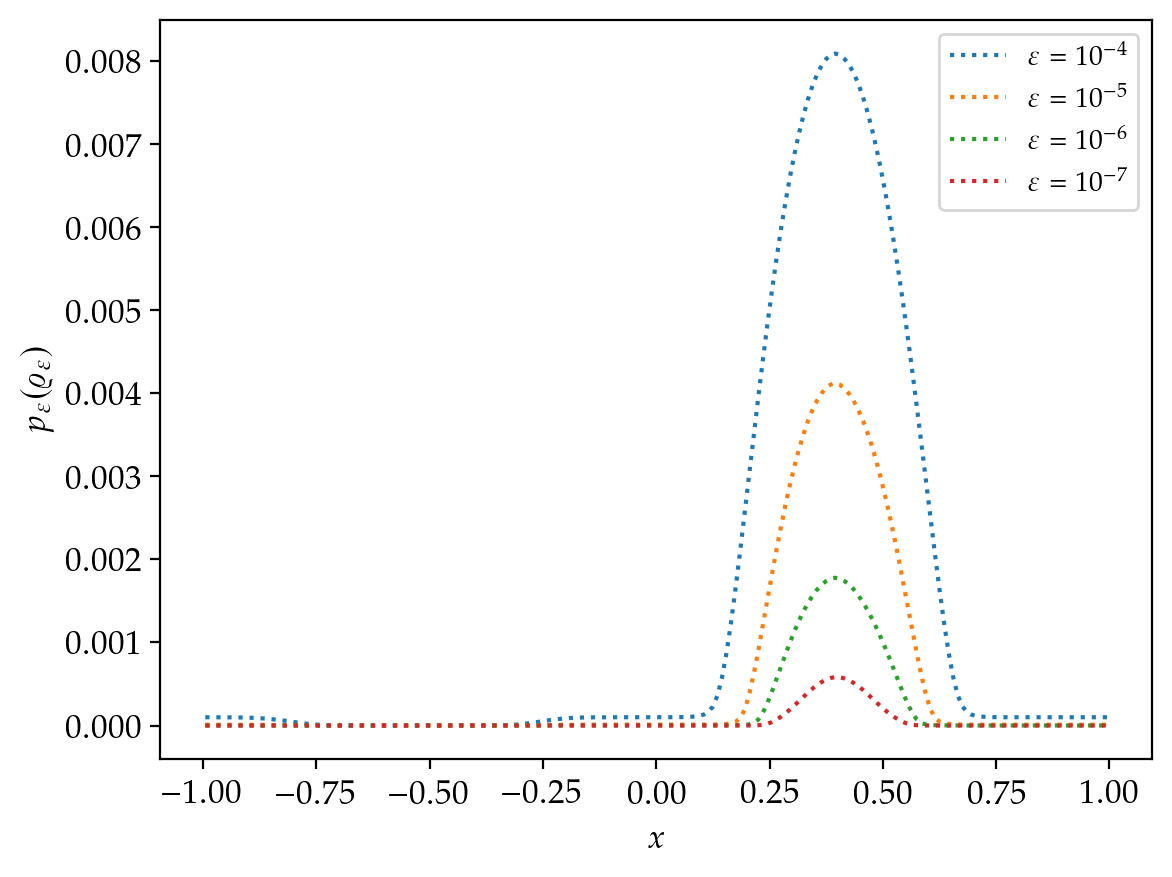}
    \caption{Example 4: the pressure $p_\veps(\vrho_\veps)$ for different values of $\veps$}
    \label{fig:pres_yang}
\end{figure}

\subsection*{Example 5}
\label{subsec:eg-5}

We consider the following two-dimensional test case from \cite{DHN11} which illustrates the collision of two congested domains. We suppose that $\Omega = \lbrack 0,1\rbrack\times\lbrack 0,1\rbrack$. The initial data is as follows:
\begin{align}
\label{eqn:eg-5}
    &\vrho_\epso(x,y) = 0.8\X_{A\cup B}(x,y) + 0.6\X_{\Omega\setminus(A\cup B)}(x,y),\\
    &q^x_\epso(x,y) = \X_A(x,y) - \X_B(x,y),\quad q^y_\epso(x,y) = 0.  
\end{align}
Here, $A,B\subseteq\Omega$ are the initial regions of congestion given by 
\[
A = \biggl\lbrack\frac{1}{6},\frac{5}{12}\biggr\rbrack\times\biggl\lbrack\frac{1}{3},\frac{7}{12}\biggr\rbrack,\quad B = \biggl\lbrack\frac{7}{12},\frac{5}{6}\biggr\rbrack\times\biggl\lbrack\frac{5}{12},\frac{2}{3}\biggr\rbrack
\]
and $\X_{C}$ denotes the indicator function of $C\subseteq\Omega$. 

We choose $\veps = 10^{-4}$ and set $M_x = M_y = 200$. In Figure \ref{fig:den-2d-cong}, we present the pseudocolor plots of the density that have the momentum vector fields superimposed onto them for different times. The collision of the moving domains causes the formation of a congested region as observed when $t = 0.05$ and $t = 0.1$. Further, the collision also causes the formation of two oppositely moving corridors, which can be clearly seen because of the superimposed quiver plots of the momentum at $t = 0.2$.

\begin{figure}[htpb]
    \centering
    \includegraphics[width = 0.9\textwidth]{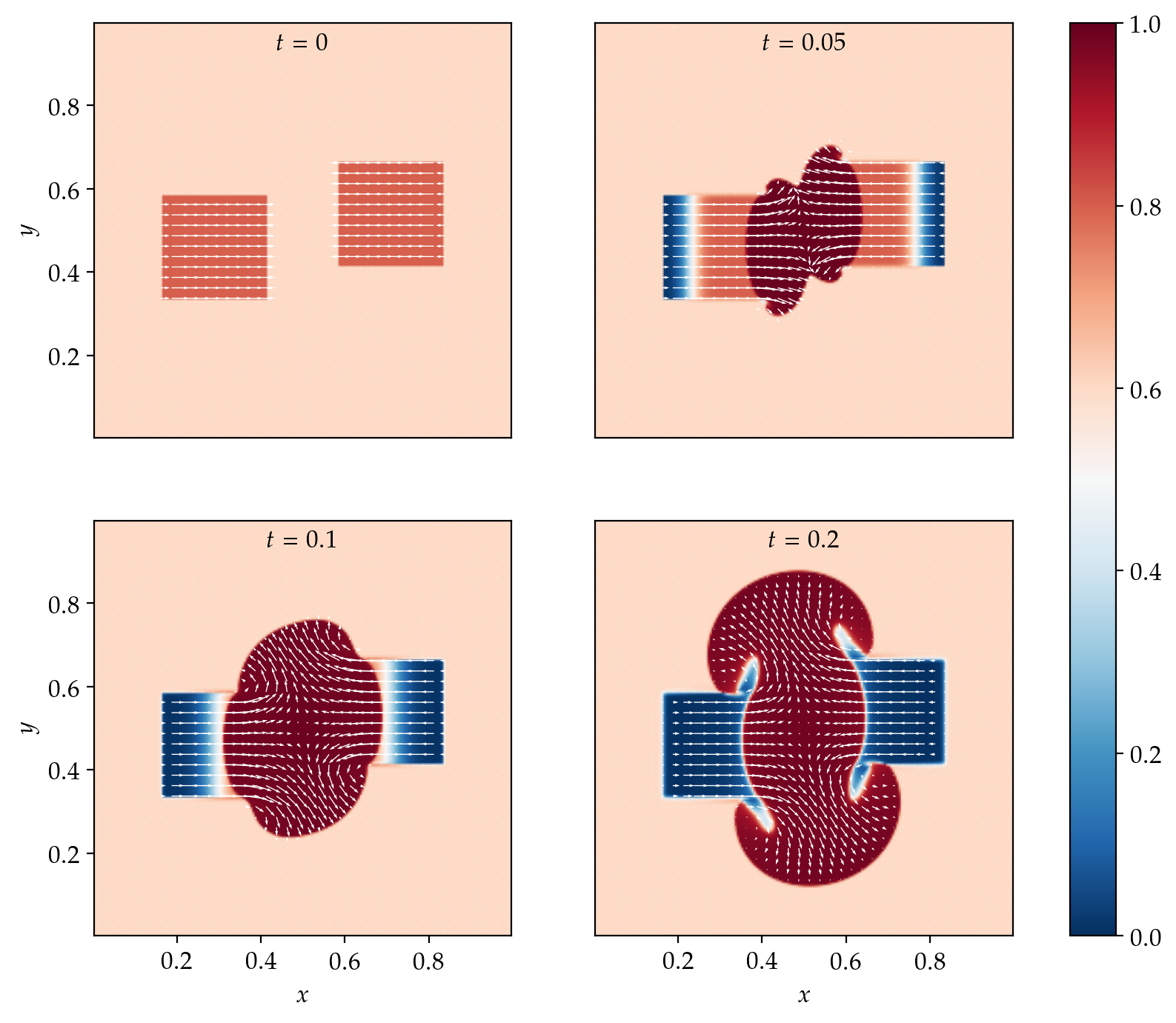}
    \caption{Example 5: momentum vector fields superimposed on top of the density pseudocolor plots for different times}
    \label{fig:den-2d-cong}
\end{figure}

\subsection*{Example 6}

We consider the following two-dimensional problem from \cite{MP17} for the presureless Euler system which illustrates the formation of a vacuum corridor. The domain is $\Omega = \lbrack -40, 40\rbrack\times\lbrack -40,40\rbrack$. Let $A = \lbrace (x,y)\in\Omega\colon x^2 + y^2 < 25\rbrace$ and $B = \lbrace(x,y)\in\Omega\colon (x+31)^2 + y^2 < 5^2\rbrace$. The initial data is as follows:
\begin{align}
    &\vrho_\epso(x) = 0.2\X_A(x) + 0.2\X_B(x) + 0.8\X_{\Omega\setminus (A\cup B)}(x) \label{eqn:eg-6-den},\\
    &q^x_\epso(x) = 2\X_B(x), \quad q^y_\epso(x) = 0. \label{eqn:eg-6-mom}
\end{align}
The setup is as follows, the smaller cloud of mass on the left, i.e.\ the mass in the region $B$, is given a rightward velocity while the larger mass occupying region $A$ is at rest initially. We set $\veps = 10^{-6}$, and let $M_x = M_y = 200$. We perform a long-time simulation and present the pseudocolor plots of the density at six different times, namely $t = 0,1,2,5,10,12$, in Figure \ref{fig:den-circ-col}. At $t = 1$, we observe the two masses colliding has lead to the formation of a congested region right where the impact has occurred. As time increases, we can clearly see that the propagation of this congested region (given in dark brown color in the pseudocolor plots) causes the formation of a vacuum corridor (given in white color in the pseudocolor plots), which in turn causes congested regions to develop along the boundary of said corridor. 

\begin{figure}[htpb]
    \centering
    \includegraphics[width = 0.9\textwidth]{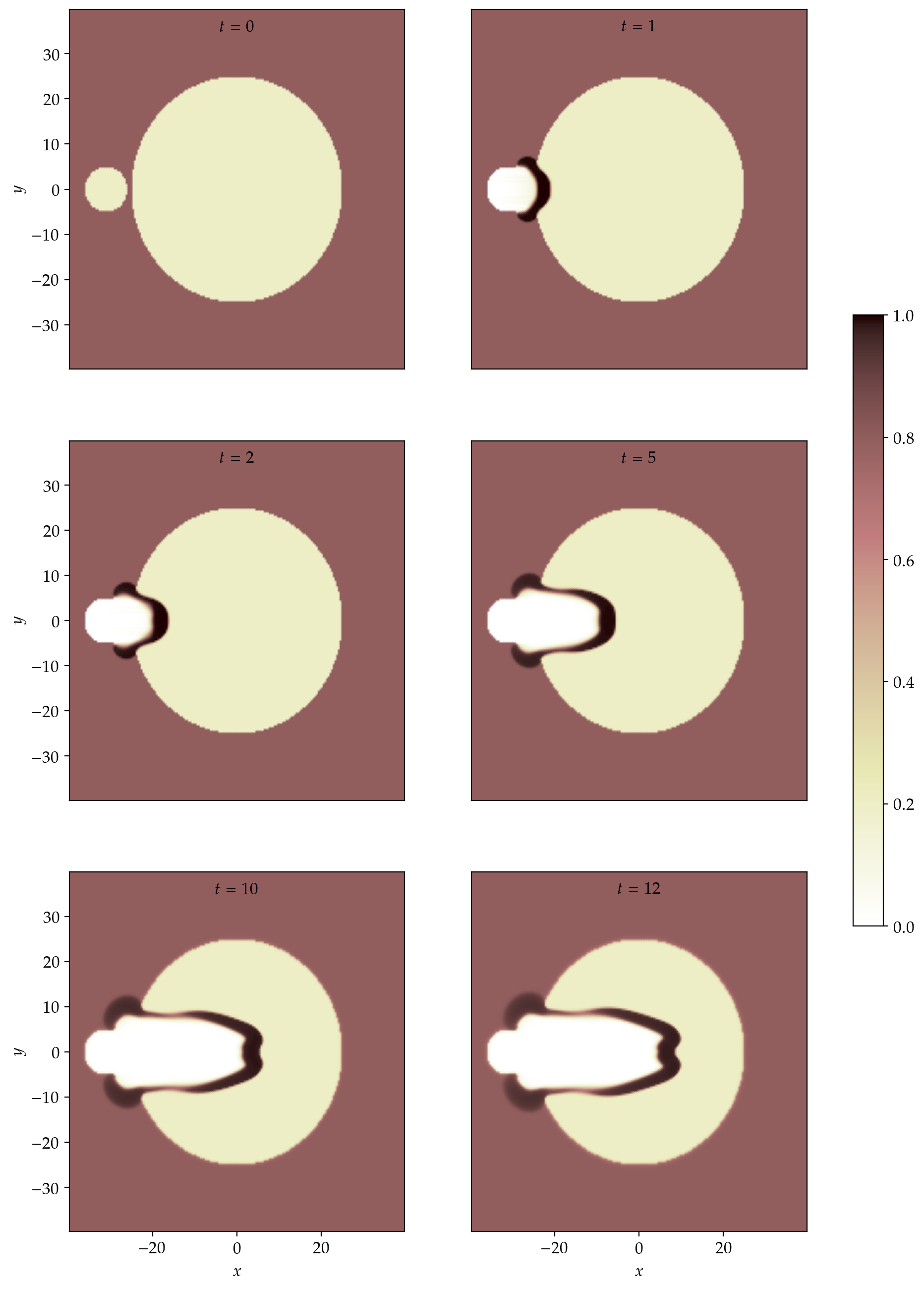}
    \caption{Example 6: pseudocolor plots of the density for different times}
    \label{fig:den-circ-col}
\end{figure}

\begin{remark}
    In comparison to \cite{MP17}, we introduce a background density in \eqref{eqn:eg-6-den} in order to make sure that the solver gets initialized properly, as initializing with the background being vacuum causes the solver to break down. However, we remark that the value of the background density does not affect outcome of the numerical experiment in any capacity.
\end{remark}

\subsection*{Example 7}
\label{subsec:eg-7}
This test case is to illustrate the behaviour of the scheme in the incompressible regime. We consider the initial data from \cite{DT11} and modify it to our purpose. The domain is $\Omega = \lbrack 0,1\rbrack\times\lbrack0, 1\rbrack$ and the initial data reads 
\begin{align}
    &\vrho_\epso(x,y) = 1 - \veps e^{-(x^2 + y^2)},\quad q^x_\epso(x,y) = \sin(2\pi(x - y)) + \veps^2\sin(2\pi(x + y)),\\
    &q^y_\epso(x,y) = \sin(2\pi(x - y)) + \veps^2\cos(2\pi(x + y)).
\end{align}
In comparison to \cite{DT11}, we have modified the initial density so that it is strictly less than 1 for each $\veps > 0$, and also to ensure that $\vrho_\epso\to 1$ as $\veps\to 0$. We set $\veps = 10^{-6}$ and also set $M_x = M_y = 200$. We consider a final time of $t = 0.02$ and present the results in Figure \ref{fig:deg-tang-prob}. In Figure \ref{subfig:deg-den}, we present the deviation of density from $1$, i.e. $1-\vrho_\veps$, and we can observe that it is near 0. In addition, we also report that the $L^1$-norm of $1 - \vrho_\veps$ is $\approx 10^{-7}$. Further, in Figures \ref{subfig:deg-vx} and \ref{subfig:deg-vy}, we present the $x$ and $y$-velocity profiles respectively and we can see that they agree with the results presented in \cite{DT11}. Thus, we can conclude that the present scheme is able to capture the solutions even in the incompressible regime.

\begin{figure}
    \begin{subfigure}{0.3\textheight}
        \includegraphics[width = \textwidth]{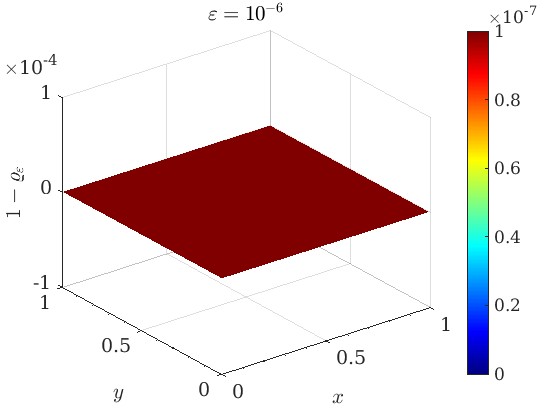}
        \caption{Deviation of density from 1}
        \label{subfig:deg-den}    
    \end{subfigure}
    \begin{subfigure}{0.3\textheight}
        \includegraphics[width = \textwidth]{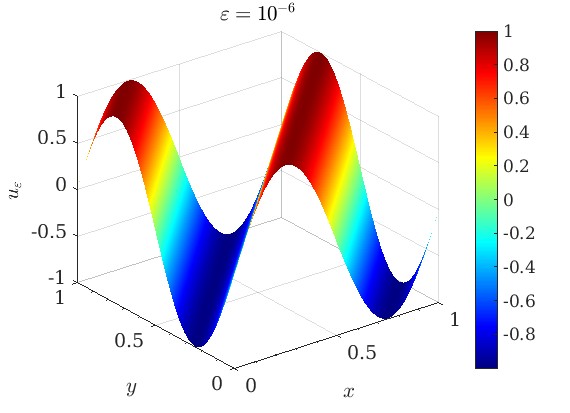}
        \caption{$x$-component of velocity.}
        \label{subfig:deg-vx}    
    \end{subfigure}
    \begin{subfigure}{0.3\textheight}
        \includegraphics[width = \textwidth]{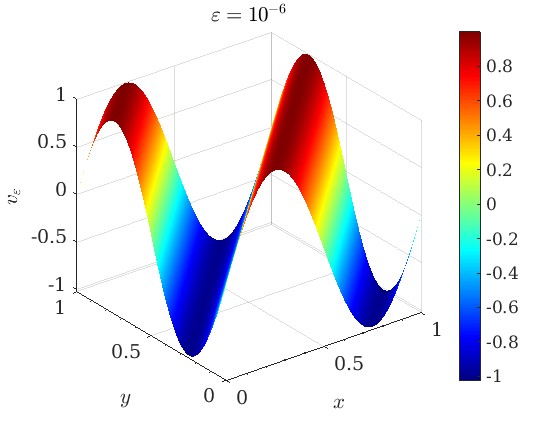}
        \caption{$y$-component of velocity.}
        \label{subfig:deg-vy}    
    \end{subfigure}
    \caption{Example 7: surface plots of the deviation of density and velocities}
    \label{fig:deg-tang-prob}
\end{figure}

\section{Conclusion}
\label{sec:conc}

In this work, we have considered the isentropic Euler system with the congestion pressure law and have designed and analyzed a finite volume scheme for the same. The choice of the pressure law imposes a constraint on the density of the form $0\leq \vrho <1$, with the pressure law also containing a small parameter $\veps$ in order to adjust the stiffness of this density constraint. The proposed scheme is semi-implicit in time and upwind in space, and it possesses key properties such as positivity of density, energy stability and further, we have also proven that the numerical densities generated by the scheme satisfy the constraint at the discrete level. We achieve the energy stability by introducing an appropriate velocity shift in the convective fluxes of the mass and momentum balances. By means of numerical case studies, the scheme is shown to be robust enough to approximate the solutions accurately for every $\veps > 0$. In addition, the AP property of the scheme is verified using specific numerical experiments, which showcases the ability of the scheme to appropriately capture the dynamics of the limit system when $\veps\sim 0$.

\bibliography{ref}
\bibliographystyle{abbrv}
\end{document}